\documentclass[11pt, a4paper, reqno, twoside, makeidx]{amsart}
\usepackage[margin=2.8cm, marginpar=2cm]{geometry}
\usepackage[dvipsnames]{xcolor}
\usepackage[font={small,it}]{caption}
\usepackage{tikz-cd}
\usepackage{amsmath, amsthm, amssymb, amsfonts, amscd}
\usepackage[english]{babel}
\usepackage{url}
\usepackage{hyperref}
\usepackage{mathtools}
\usepackage{pinlabel}
\usepackage{changepage}

\usepackage{amsaddr}
\theoremstyle{remark}
\newtheorem{rmk}{Remark}[section]
\theoremstyle{plain}
\newtheorem{thm}{Theorem}[section]
\newtheorem{lem}[thm]{Lemma}
\newtheorem{prop}[thm]{Proposition}
\newtheorem{defi}[thm]{Definition}
\newtheorem{cor}[thm]{Corollary}
\newtheorem{conj}[thm]{Conjecture}

\theoremstyle{definition}
\newtheorem{exa}[thm]{Example}

\newtheorem*{cccc}{Conjecture}

\newcommand{\N}{\mathbb{N}}
\newcommand{\Z}{\mathbb{Z}}

\newcommand{\R}{\mathbb{R}}

\newcommand{\F}{\mathbb{F}}
\newcommand{\HH}{\mathrm{H}}
\newcommand{\Mat}{\mathrm{M}}
\newcommand{\uH}{\ddot{\mathrm{H}}}
\newcommand{\hasse}{\mathcal{P}(X)}
\newcommand{\perc}[2]{#1 \cdot 10^{-#2}}

\title{Filtered simplicial homology, graph dissimilarity and \"{u}berhomology}
\author{Daniele Celoria}
\date{}

\begin{document}

\maketitle

\begin{abstract}
We introduce a filtration on the simplicial homology of a finite simplicial complex $X$ using bi-colourings of its vertices. This yields two dual homology theories closely related to discrete Morse matchings on $X$. We give an explicit expression for the associated graded object of these homologies when $X$ is the matching complex of the Tait graph of a plane graph $G$, in terms of subgraphs determined by certain matchings on the dual of $G$. 

We then use one of these homologies, in the case where $X$ is a graph, to define a conjecturally optimal dissimilarity pseudometric for graphs;  we prove various results for this dissimilarity and provide several computations.

We further show that, by organising the horizontal homologies of a simplicial complex in the poset of its colourings, we obtain a triply graded homology theory which we call \"uberhomology. This latter homology is not a homotopy invariant, but nonetheless encodes both combinatorial and topological information on $X$. For example, we prove that if $X$ is a subdivision, the \"uberhomology vanishes in its lowest degree, while for an homology manifold it coincides with the fundamental class in its top degree.
We compute the \"uberhomology on several classes of examples and infinite families, and prove some of its properties; namely that, in its extremal degrees, it is well-behaved under coning and taking suspension.

We then focus on the case where $X$ is a simple graph, and prove a detection result.  Finally, we define some singly-graded homologies for graphs obtained by specialising the \"uberhomology in certain bi-degrees,  provide some computations and use computer aided calculations to make some conjectures.
\end{abstract}

\section{Introduction}
Since its definition in the beginning of the previous century by Poincar\'e (and after some key insights by Noether in 1925 \cite{hilton1988brief}) simplicial homology has been an invaluable tool in mathematics, well beyond the mere study and classification of combinatorial manifolds.

In the last few decades, simplicial homology has seen a new renaissance in at least two new prominent fields: the definition and efficient computation of persistent homology in topological data analysis \cite{edelsbrunner2008persistent}, and the study of properties of simplicial complexes associated to graphs, using techniques from combinatorial algebraic topology \cite{kozlov2007combinatorial}. \\

The aim of this paper is to show that, by considering a simple combinatorial filtration on the simplicial chain complex of a simplicial complex $X$, it is possible to define various interesting homology theories. 

More concretely, a black and white colouring of the vertices of a finite simplicial complex $X$ induces a filtration on the simplicial chain complex of $X$, as well as a splitting of the differential $$\partial = \partial_h + \partial_d.$$ 

This allows for the definition of a filtered chain complex, together with two auxiliary bi-graded homology theories $\HH^h(H)$ and $\HH^d(X)$, called horizontal and diagonal homology.

Opposite colourings give rise to an involutive relation between $\HH^h$ and $\HH^d$. We will thus mainly restrict our considerations only to the horizontal version.\\

After proving that they are well-defined, we show that these homology theories are related to certain discrete Morse matchings (as defined in \cite{chari2000discrete} following \cite{forman1998morse}) on the face poset of $X$. More precisely, we completely characterise the colourings that give rise to Morse matchings on $X$ in Theorem~\ref{thm:dmmandsubgraphs}, and prove that these can be used as building blocks for the case of a general colouring in Theorem~\ref{thm:decomposition}. This allows to canonically partition the simplicial differential in a collection of discrete Morse matchings.
As a consequence, in some cases the generators of the horizontal homology are in bijection with the critical cells of the Morse matchings induced by the colouring; further, for these specific colourings we can give an explicit description of the generators of the horizontal homology (Proposition~\ref{prop:dalmatianhor}).\\
A further implication of Theorem~\ref{thm:decomposition} is that the matrices determining the horizontal complex can be easily computed by just looking at the contribution of rather simple discrete Morse matchings, that can be easily explicitly determined (see Section~\ref{sec:relationtodmm}, and Figure~\ref{fig:simplexcomplex} therein).\\

On a different note, we show that the matching complex of a Tait graph $\Gamma(G)$ of a plane graph $G$ can be given a natural bi-colouring. The homology of the filtered chain complex with respect to this colouring is a page of a spectral sequence which starts from the homology of the matching complex of the barycentric subdivision of $G$, and abuts to $\HH (\Mat (\Gamma(G)))$. We show in Theorem~\ref{thm:horizhomoftait} that the horizontal homology $\HH^h(\Mat(\Gamma(G)))$, which is identified with the associated graded object to the filtration, admits a decomposition in terms of subgraphs in the barycentric subdivision of $G^*$, the plane dual of $G$. This decomposition allows to straightforwardly compute the homology in terms of the simplicial homology of matching complexes of certain complementary subgraphs in  the barycentric subdivision of $G$. Thus, Theorem~\ref{thm:horizhomoftait} reveals a rich structure of the matching complex of Tait graphs, tightly connecting properties of the initial graph $G$ to $\Mat(\Gamma(G))$.\\

We then focus our attention to the general computation of the horizontal homology of simple and connected graphs. We partition these homologies with respect to the number of vertices coloured in black; these collections, denoted by $\Theta(G,j)$ for $j = 0, \ldots, m$, can be used to define a dissimilarity measure $\Delta$ for graphs. 
A dissimilarity measure is just a function assigning a value from $0$ to $1$ to any pair of graphs with the same number of vertices, that somehow quantifies their being ``structurally different'' (see \emph{e.g.}~\cite{schieber2017quantification}, \cite{chartrand1998graph} and \cite{li2018efficient}  for similar measures, and \cite{wills2020metrics} for an interesting up-to-date survey). Common (dis)similarity measures include feature-based distances, such as connectivity, degrees and cycle sequences, or spectral distances, extracted from the graph's adjacency or Laplacian matrices.
We note here that this dissimilarity measure admits a straightforward generalisation to the case of simplicial complexes, but we are not pursuing this direction presently.

One can think of the dissimilarity $\Delta$ as a systematic method for collecting and organising several feature-based measures in one object.
We show that $\Delta$ can pick up some subtle information on the graph.

Similarly, there are several metrics that can be considered on the space of graphs. Notable examples include \cite{balavz1986metric}, \cite{bhutani2003metric}, and \cite{bento2018family}; it is interesting to note that most of these examples are somewhat related to applications.
It is not hard to see that $\Delta$ is a pseudometric on the set of graphs with a fixed number of vertices. Indeed, we conjecture the following:
\begin{cccc}
The dissimilarity $\Delta$ is a metric on the set of graphs, \emph{i.e.}~$\Delta(G_1,G_2) > 0$ whenever $G_1 \neq G_2$. 
\end{cccc}
It turns out that this conjecture is equivalent to saying that the $\Theta$ invariants introduced above are in fact complete invariants of graphs.

We then discuss the computational cost involved in determining $\Delta$, as well as some effective techniques to reduce it. 
We provide the result of the  computation of $\Delta$ for all connected simple graphs with less than $10$ vertices and all trees with up to $16$ vertices. A program for computing $\Delta$ and related invariants is available at~\cite{miogithub}.\\

We can then introduce the main object of the paper, the aptly named \"uberhomology $\uH(X)$; this is a triply graded homology theory associated to any finite simplicial complex $X$. We obtain $\uH(X)$ by considering at once the horizontal homologies for all possible colourings of $X$, and organising them in a poset reminiscent of Khovanov's \cite{khovanov2000categorification} cube of resolution for link diagrams.

It is not hard to realise that the \"uberhomology is not a homotopy invariant.
We can however show that the \"uberhomology in its extremal degrees measures both combinatorial and topological features of simplicial complexes; as an example we prove in Theorem~\ref{thm:0dim} that the $0$-degree component of this homology keeps track of the intersections of the radius one closed balls centred in the vertices of the simplicial complex. 

In particular this implies (Corollary~\ref{cor:subdivisiondegree0}) that $\uH^0(X)$ vanishes whenever $X$ is the barycentric subdivision of another simplicial complex, or has diameter greater than $3$ (Corollary~\ref{cor:diameter}).

On the other hand, the top dimensional \"uberhomology group is a homotopical invariant of $X$, whenever $X$ is a triangulation of a closed manifold (Theorem~\ref{thm:topdim}). In this case it coincides with the fundamental class of $X$. 
We also show in Proposition~\ref{prop:cone} that the \"uberhomology in these extremal degrees behaves well under the coning and suspension operations.
A general combinatorial or topological characterisation of $\uH(X)$ in the intermediate degrees seems to be currently out of reach. We do however fully compute $\uH(X)$ in several examples, such as the $n$-simplices and their boundary. 

We conclude by showing that the \"uberhomology can be used in several ways to study graphs; as a straightforward example, we can compute the \"uberhomology of a simple graph, seen as a simplicial complex. Considering the degrees in which $\uH$ is non-trivial yields four singly-graded homology theories --one of which appears to be trivial-- with unusual properties. These singly-graded graph homologies obtained by specialising the \"uberhomology might be of independent interest.

Some Sage \cite{sagemath} programs to compute these homologies are available in~\cite{miogithub}.
.

Furthermore, we briefly examine the \"uberhomology of simplicial complexes associated to graphs, \emph{e.g.}~by considering the horizontal homologies of its matching complex, coloured by the subgraphs poset.

\subsection*{Acknowledgements:} The author wishes to thank  N.~Yerolemou, A.~Barbensi, M.~Golla and N.~Scoville for interesting conversations and feedback. The author would also like to express sincere thanks to anonymous referees, which greatly improved the earlier version of this paper. This project has received funding from the European Research Council (ERC) under the European Union’s Horizon 2020 research and innovation programme (grant agreement No 674978).



\section{Horizontal and diagonal homologies}

Let $X$ be a finite and connected simplicial complex, with ordered vertex set $V(X) = \{v_1, \ldots, v_m\}$. Fix $\varepsilon \in \Z_2^m$, and colour the vertex $v_i$ white or black according to whether $\varepsilon(i)$, the $i$-th component of $\varepsilon$, is $0$ or $1$ respectively. The pair $(X,\varepsilon)$ will be called an \emph{$\varepsilon$-coloured simplicial complex}. 

Given a $n$-dimensional simplex $\sigma \subseteq X$, denote by $$ w_\varepsilon(\sigma) = n+1-\sum_{v_i \in \sigma} \varepsilon(i)$$  the \emph{weight} of $\sigma$ with respect to the colouring $\varepsilon$; here $\varepsilon(i) \in \Z_2$ is the $i$-\emph{th} component of the colouring. In other words, $w_\varepsilon (\sigma)$ is just the number of white vertices contained in $\sigma$. Likewise $w_\varepsilon (X)$ is defined as $\max\{w_\varepsilon(\sigma)|\sigma \subseteq X\}$.  We will write $dim(\sigma)$ to denote the dimension of a simplex, and $dim(X)$ for the maximal dimension of the simplices in $X$.

If $\{ v_1, \ldots,v_n\} \subseteq V(X)$, then $\langle v_1, \ldots,v_n\rangle$ indicates the simplex in $X$ spanned by these vertices.
The link $lk(v)$ of a vertex $v \in V(X)$ is the subcomplex $\{\sigma \in X\,|\, \langle v\rangle\cap \sigma = \emptyset,\,\sigma \cup \langle v\rangle \in X \}$, and the star $St(v)$ is the set of simplices having $v$ as a vertex; its closure\footnote{This is just the union of the star and link of the vertex.} will be denoted by $\overline{St(v)}$.
~\\

In what follows the simplicial homology of $X$, with coefficients in the field $\F$ with two elements, will be denoted by $\HH_*(X) = \HH (C_*(X), \partial)$, where $$C_i(X) = \F\langle\,\sigma \;|\; dim(\sigma) = i\rangle,$$ and $\partial$ is the simplicial boundary map.  The notation $\widetilde{\HH}(X)$ will instead be reserved for the reduced homology of $X$, obtained from $(C_*(X), \partial)$ by adding one extra generator $\star$ to $C(X)$ in homological degree $-1$, and setting  $\partial\langle v\rangle = \star$ for each vertex $v \in V(X)$.\\ 

If we fix the colouring $\varepsilon$, the weight $w_\varepsilon$ induces a filtration on the complex $C_*(X)$:
$$C_i(X,\varepsilon, \le k) = \F\langle\; \sigma\;|\; \sigma \in C_i(X),\;  w_\varepsilon(\sigma) \le k\rangle.$$ 
Clearly we have 
\begin{equation}
C_*(X,\varepsilon, \le k) = \begin{cases}
\emptyset &  \text{ if } k < 0\\
C_*(X) & \text{ if } k \ge dim(X) +1
\end{cases}
\end{equation}
since, by definition, $0 \le w_\varepsilon(\sigma) \le dim(X) +1$ for all $\sigma \subseteq X$.

We will adopt the notation $C(X, \varepsilon)$ to denote the module $C(X)$ endowed with the bi-grading induced by the $\varepsilon$ weight function and the usual dimension of the simplices, and $$C_i(X, \varepsilon,k) = \F \langle \; \sigma\;|\; \sigma \in C_i(X),\;  w_\varepsilon(\sigma) = k\rangle$$ to denote the  complex in bi-degree $(i,k)$, so that $$C(X, \varepsilon) = \bigoplus_{\substack {i = 0, \ldots,\, dim(X) \\ k = 0, \ldots,\, dim(X)+1}} C_i(X, \varepsilon,k).$$ 
~\\
The simplicial differential can be split into two components
$$\partial = \partial_h + \partial_d,$$
where $\partial_h$ preserves the weight $w_\varepsilon$, while $\partial_d$ decreases it by $1$. 

More explicitly, the \emph{horizontal differential} $\partial_h (\sigma)$ is the sum of all the faces in the expansion of $\partial(\sigma)$ that contain exactly $w_\varepsilon(\sigma)$ white dots, while the \emph{diagonal differential} $\partial_d (\sigma)$ is the sum of the remaining faces, which necessarily have one fewer white dot (see Figures~\ref{fig:colouredcpx2simplex} and~\ref{fig:hor3simplex}).

\begin{lem}
Given a finite $\varepsilon$-coloured simplicial complex $(X, \varepsilon)$, both $\partial_h$ and $\partial_d$  square to $0$ in $C(X, \varepsilon)$.
\end{lem}
\begin{proof}
For a coloured simplex $\sigma \subseteq (X, \varepsilon)$, $\partial_h^2(\sigma)$ is the (possibly empty) sum over all codimension two faces obtained by removing two black vertices from $\sigma$. Note that if $\sigma$ contains fewer than two black vertices, the square of $\partial_h$ is automatically $0$. 
Each of these faces can be obtained in exactly two ways, by removing one  black vertex at a time in either order.  Hence the total sum vanishes modulo $2$. The same holds for $\partial_d$ by exchanging the role of the colours.
\end{proof}

An explicit computation of these homologies for the $(1,0,0)$ $2$-simplex is carried out in Figure~\ref{fig:colouredcpx2simplex}. The whole filtered chain complex $C(\Delta^2,(1,0,0))$ is in the top-right part of the figure, while the horizontal and diagonal complexes are shown in the lower part. We note that the horizontal and diagonal differentials are obtained by only considering the horizontal\slash diagonal arrows representing $\partial$, thus justifying their name.
The homologies of the latter two are  $\HH^h(\Delta^2, (1,0,0)) \cong \F\langle v_1\rangle$ in bi-degree $(0,0)$ and $\HH^d(\Delta^2, (1,0,0)) \cong \F\langle v_2\rangle$ in bi-degree $(0,1)$.

\begin{figure}[ht]
\includegraphics[width=10cm]{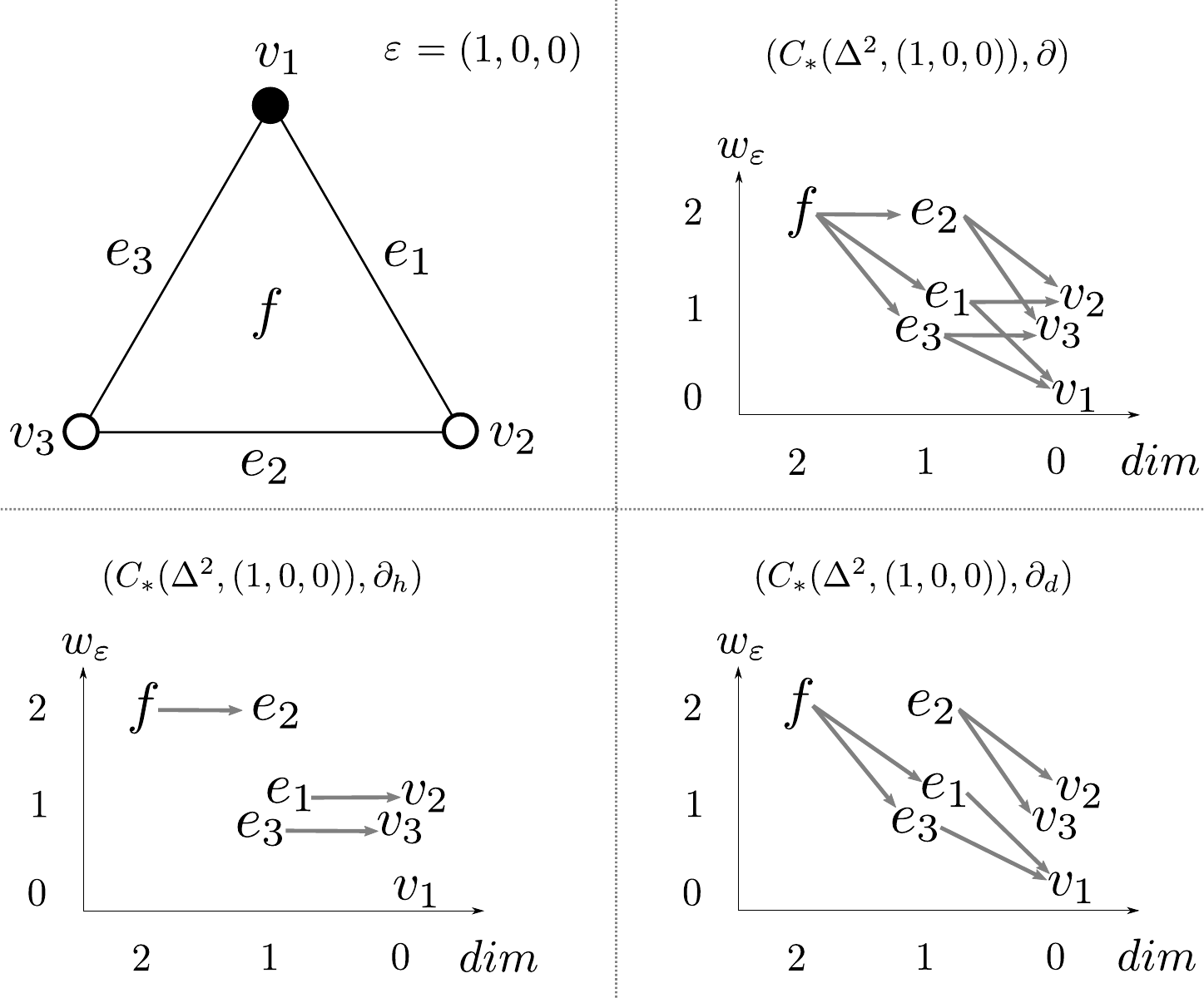}
\caption{A $(1,0,0)$-coloured $2$-simplex, with the bi-graded chain complexes described in Definition~\ref{def:hordiag} associated to it.}
\label{fig:colouredcpx2simplex}
\end{figure}
\begin{figure}[ht]
\includegraphics[width=10cm]{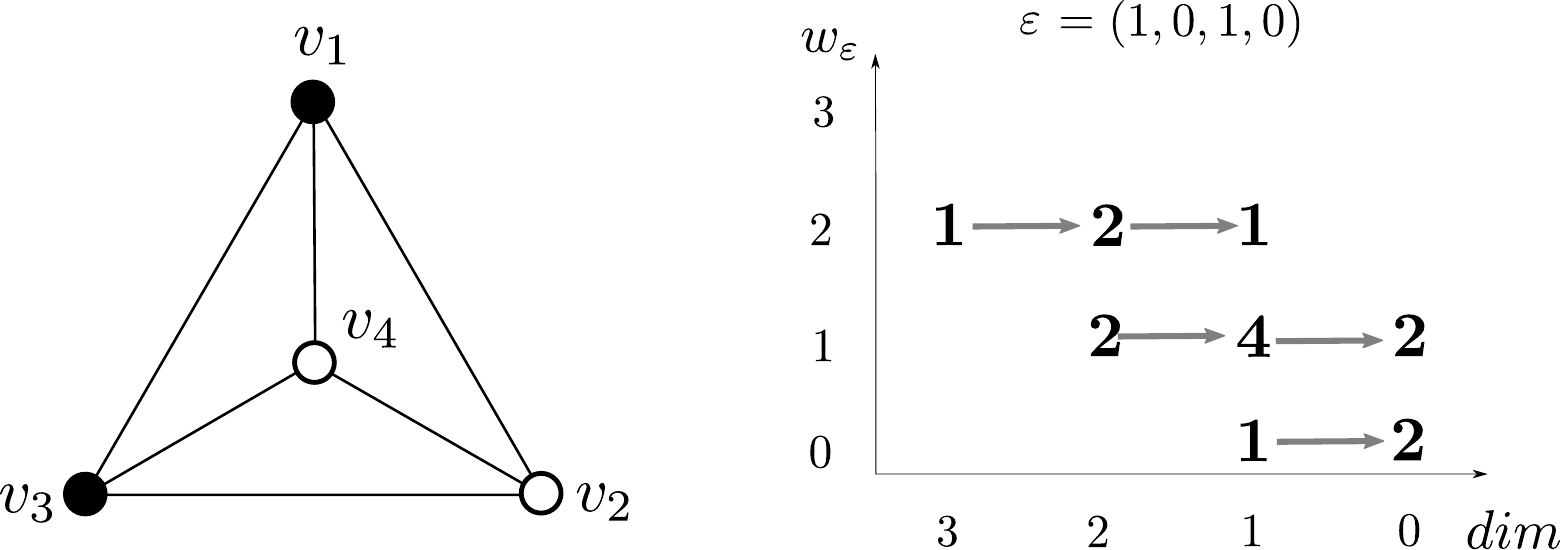}
\caption{The bi-coloured $3$-simplex $(\Delta^3, (1,0,1,0))$ and the corresponding horizontal complex $\left(C_*(X,(1,0,1,0),*), \partial_h\right)$, where the numbers in the table denote the (non-zero) ranks of the complex in that bi-grading. The differential is depicted with (horizontal) arrows. The total homology $\HH^h(\Delta^3, (1,0,1,0))$ (which can be easily obtained by explicitly computing the differentials) is $\F\langle v_1\rangle$ in bi-degree $(0,0)$.}
\label{fig:hor3simplex}
\end{figure}

\begin{defi}\label{def:hordiag}
The \emph{$\varepsilon$-horizontal homology} $\HH^h(X,\varepsilon)$ is the homology of the bi-graded complex $\left(C(X,\varepsilon) , \partial_h \right)$, and the  
\emph{$\varepsilon$-diagonal homology} $\HH^d(X,\varepsilon)$ is the homology of the bi-graded complex $\left(C(X,\varepsilon)  , \partial_d \right)$.

In particular, we can write
\begin{equation}
\HH^h(X,\varepsilon) = \bigoplus_{\substack{ i = 0, \ldots,\, dim(X) \\k = 0, \ldots,\, dim(X)+1}} \HH^h_i(X,\varepsilon,k)
\end{equation}
where each summand on the right is in bi-degree $(dim, w_\varepsilon) = (i,k)$.

As the simplicial differential $\partial$ maps $C_*(X,\varepsilon,  \le k)$ into itself,  it is also possible to define a further sequence of homology groups $\mathcal{H}_*(X,\varepsilon,k)$, defined as $\HH\left(C_*(X,\varepsilon,  \le k), \partial \right)$. 
\end{defi}

Note that the $\varepsilon$-horizontal homology of $X$ is just the homology of the associated graded object\footnote{The associated graded object is the homology of the quotient of two consecutive filtration levels, hence it only has the grading-preserving part of the differential, which is $\partial_h$ in our case.} for the filtered complex $\left(C_*(X,\varepsilon,  \le k), \partial \right)$. Moreover, changing the ordering of the vertices in $X$ does not change the homology, up to isomorphism.

\begin{rmk}\label{rmk:isohordiag}
As inverting the colour of all vertices and exchanging the horizontal and diagonal differentials yields  isomorphic chain complexes, we have $$\HH^d_*(X, \varepsilon) \cong \HH^h_*(X,\overline{\varepsilon}),$$ where $\overline{\varepsilon}$ is the unique complementary colouring such that $\overline{\varepsilon}(i) + \varepsilon(i) = 1$ for all $i = 1, \ldots, m$. 
\end{rmk}

It is possible to state explicitly the relation between $\HH^h$, $\HH^d$ and simplicial homology; indeed, note that in the position corresponding to $\varepsilon = (0,\ldots, 0)$ we get the natural isomorphism
\begin{equation}\label{eqn:zeroepsilon}
\HH^h_*(X,(0,\ldots, 0)) \cong C_*(X,(0,\ldots, 0))  \cong \bigoplus_{k \ge 0} C_k(X)[k+1].
\end{equation}
The notation $C_*(X)[k]$ indicates a complex (with trivial horizontal differential) concentrated in filtration degree $k$; this degree correction in Equation~\eqref{eqn:zeroepsilon} is required by the fact that  for this trivial colouration every simplex $\sigma$ in $X$ contains exactly $dim(\sigma) +1$ white vertices.

On the other side, we have that 
\begin{equation}\label{eqn:oneepsilon}
\HH^h_*(X,(1,\ldots, 1)) \cong \HH_*(X).
\end{equation} 
In this case we do not need any shift, as the whole complex is contained in $w$-degree $0$. Using Remark~\ref{rmk:isohordiag}, it is immediate to see that the opposite holds for the diagonal homology; namely we obtain an isomorphism with simplicial homology for the all $0$ colouring, and with the (suitably shifted) simplicial chain complex for the all $1$ colouring. 

Moreover, if we take a sequence of colourings $\{\varepsilon^p\}_{p = 1, \ldots,m}$ starting from $(0,\ldots, 0)$, ending in $(1, \ldots, 1)$ and increasing $|\varepsilon|$ at each step, it follows from Equations~\eqref{eqn:zeroepsilon} and \eqref{eqn:oneepsilon} that if we consider the sequence of groups obtained from the horizontal homologies at each colour $\varepsilon^p$ by ``forgetting'' the filtration (see Definition~\ref{def:flattening}),  we get a sequence of groups  ``interpolating'' between $ C_*(X)$ and  $\HH_*(X)$. 

We will see in Section~\ref{sec:relationtodmm} that when passing from one colour in the sequence to the next, the horizontal differential is ``enriched'' with arrows related to certain elementary discrete Morse matchings (\emph{cf.}~Theorem~\ref{thm:decomposition}).\\

Given a $\varepsilon$-coloured simplicial complex, we can consider the polynomial obtained as graded Euler characteristic of the horizontal homology.
\begin{defi}\label{def:decateghoriz}
For a coloured simplicial complex $(X, \varepsilon)$ let us define  $$\chi(X,\varepsilon) = \sum_{i = 0}^{dim(X)} \sum_{k = 0}^{dim(X)+1} (-1)^i \cdot \mathrm{rk}(\HH^h_i(X, \varepsilon,k)) \cdot t^k \in \Z [t]$$
This is just the sum of the Euler characteristics of the homologies $\HH^h(X, \varepsilon,k)$, weighted by their filtration degree through multiplication by $t$.
\end{defi}

The evaluation in $t = 1$ is independent of the colouring, namely (\emph{cf.}~Equation~\eqref{eqn:oneepsilon}):
\begin{equation*}
\chi(X, \varepsilon)(t = 1) = \chi(X, (1, \ldots,1))) = \chi(X),
\end{equation*}
where the last item is the usual Euler characteristic of $X$, while the central one is the polynomial in the only colour where all vertices are black. The first equality follows from the fact that by setting $t=1$, we are disregarding the information provided by the filtration.

Similarly, if we evaluate $\chi(X, \varepsilon)$ in $t = 0$, we get the Euler characteristic of the simplicial subcomplex generated by the black vertices: $\chi(X, \varepsilon) = \chi(Bl(X, \varepsilon))$ where $Bl(X, \varepsilon)$ is the simplicial complex $\langle \sigma \in X \,|\, w_\varepsilon (\sigma) = 0 \rangle$.

\section{Relation to discrete Morse matchings}\label{sec:relationtodmm}

Given a finite simplicial complex $X$, denote by $\hasse$ its face poset; this is the oriented graph representing the poset structure\footnote{We adopt the convention of not including the empty set among the vertices of $\hasse$.} of the simplices in $X$, with edges oriented in the direction of decreasing dimension (see Figure~\ref{fig:fromdifftohasse}). 

An alternative point of view is to consider $\hasse$ as a graphical representation of the simplicial chain complex for $C(X)$. Here edges between simplices in adjacent dimensions represent differentials (again, recall that by working over the field $\F$ with two elements we are disregarding signs) between its endpoints. In what follows we will consider both these interpretations simultaneously. 

\begin{defi}
Given a finite graph $G$, a \emph{matching on $G$} is a collection of disjoint edges in $G$. The collection of matchings on $G$ will be denoted by $Match(G)$. 

If $X$ is a finite simplicial complex, its  \emph{matching complex} $\Mat (X)$ is the simplicial complex spanned by \emph{matchings} on $\hasse$, \emph{i.e.}~subsets of disjoint edges in $\mathcal{P}(X)$. We will sometimes improperly refer to $m \in Match(\hasse)$ as a matching on $X$.
\end{defi}

Given a matching $m$ on $X$, we can invert the orientation of all of its edges in the face poset. If this operation does not create any oriented cycle in $\hasse$ we call $m$ a  \emph{(discrete) Morse matching}; these kinds of matchings consist of equivalence classes of Forman's~\cite{forman1998morse} discrete Morse functions (see also~\cite{chari2000discrete}). More concretely, a Morse matching can be thought of as being a discrete analogue of the gradient vector field of a Morse function on a smooth manifold. 

Simplices that are not matched by a Morse matching $m$ are called \emph{critical} for $m$. We will say that $m$ is \emph{perfect} if it matches all vertices of $\hasse$, and \emph{maximal} if it is not a proper subset of another Morse matching. Furthermore, if a matching $m$ comprises $k$ edges, we will write $|m| = k$. 

\begin{defi}
Let $(X, \varepsilon)$ be a coloured simplicial complex with $m$ vertices. Then the horizontal differential induces a subgraph $I(X, \varepsilon)$ of the face poset $\mathcal{P}(X)$, whose edges are the non-zero components of $\partial_h$. More precisely, one edge connecting two simplices $\sigma \supset \tau$ in $\mathcal{P}(X)$ will belong to $I(X, \varepsilon)$ if and only if $\tau$ appears with a non-zero coefficient in $\partial_h (\sigma)$.

We call \emph{elementary} all subgraphs of $\mathcal{P}(X)$ of the form $I(X,\varepsilon_i)$, where $\varepsilon_i \in \Z_2^m$ is non-zero only in its $i$-\emph{th} entry. These are all the subgraphs induced by colouring in black a single vertex of $X$. We will also call elementary the colourings with exactly one black vertex.
\end{defi}
\begin{figure}[ht]
\includegraphics[width=8cm]{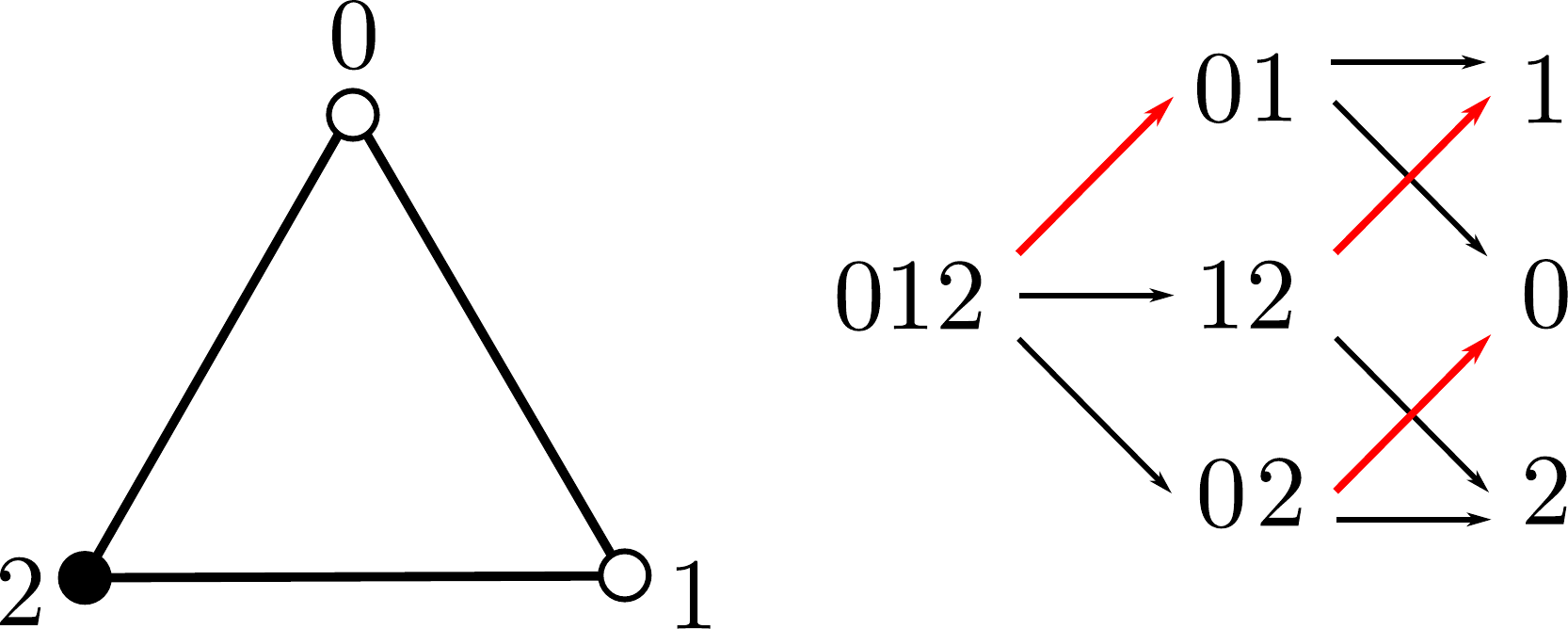}
\caption{On the left, a $\varepsilon_2$-coloured $2$-simplex; on the right the chain complex of its simplicial homology seen as a face poset, with edges of the elementary subgraph $I(\Delta^2, \varepsilon_2)$ in red.}
\label{fig:fromdifftohasse}
\end{figure}

We are going to prove below that all elementary subgraphs are in fact discrete Morse matchings on $X$; however, not all such matchings arise via this construction. 

We can nonetheless give a simple and complete  combinatorial characterisation for a colouring $\varepsilon$ guaranteeing that the subgraph $I(X,\varepsilon)$ induces a discrete Morse matching on $X$.

\begin{defi}\label{def:dalmatian}
Let $\varepsilon \neq (0, \ldots,0)$ be a colouring of the vertices of $X$. We say that $\varepsilon$ is \emph{dalmatian} if for each pair of simplices $\sigma, \tau \subseteq X$, $\sigma \cap \tau \neq \emptyset $ implies that $ V(\sigma \cup \tau)$ contains at most one black vertex. Being dalmatian is equivalent to requiring that the closures of the stars of the black vertices are pairwise disjoint.
\end{defi}

\begin{figure}[ht]
\centering
\includegraphics[width=10cm]{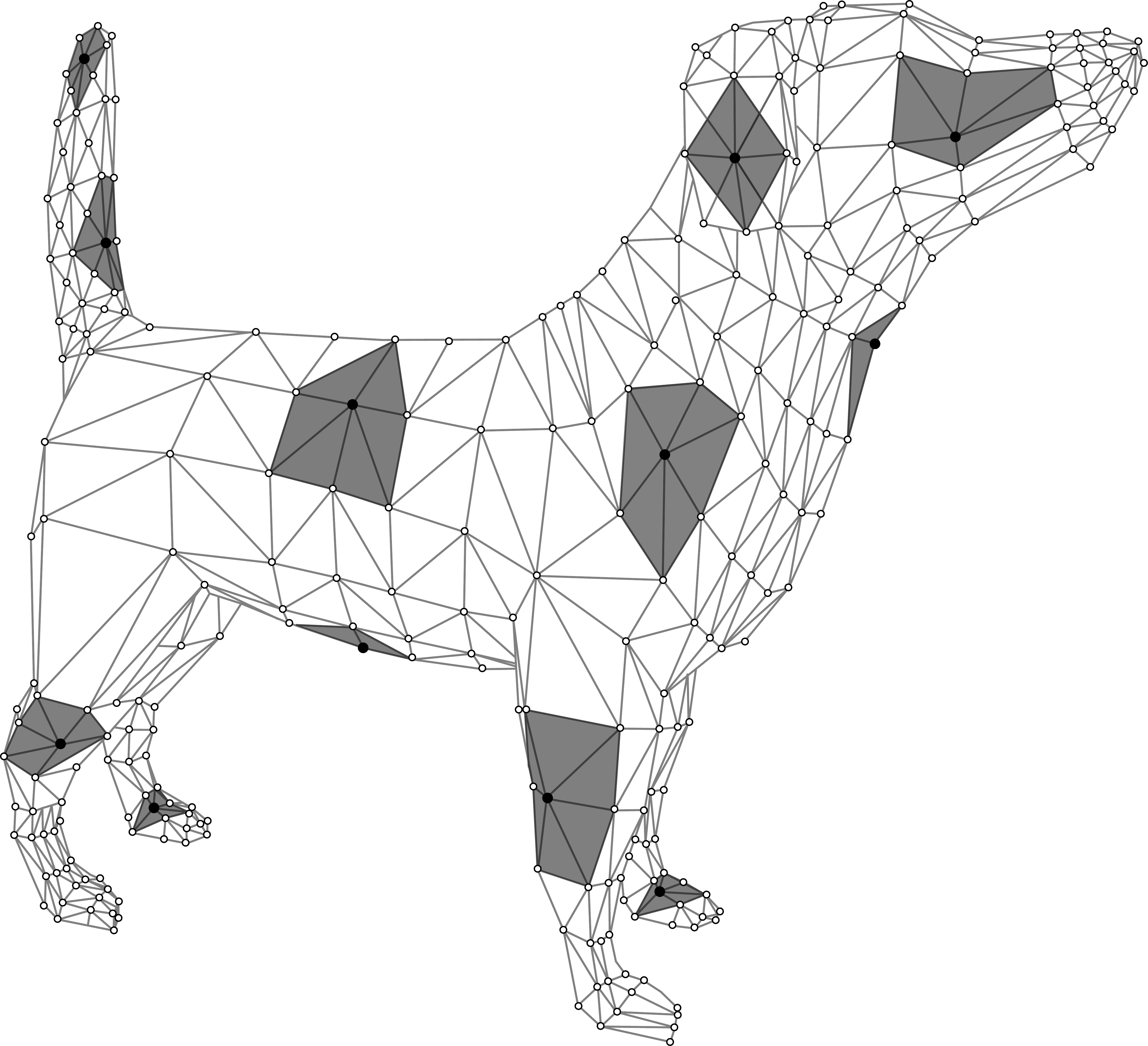}
\caption{A triangulation $X$ of $S^2$, exhibiting a dalmatian colouration. The shaded simplices are the closure of the stars of the black vertices. Picture based on \cite{dalmata}.}
\label{fig:dalmata}
\end{figure}

\begin{thm}\label{thm:dmmandsubgraphs}
The subgraph $I(X,\varepsilon) \subseteq \hasse$ is a discrete Morse matching if and only if the colouring $\varepsilon$ is dalmatian.
\end{thm}
\begin{proof}
$\Longrightarrow )$ Consider a pair of simplices $\sigma, \tau \subseteq X$ intersecting non-trivially. Assume by contradiction that $V(\sigma \cup \tau)$ contains at least two black vertices. We can distinguish two cases, depending on whether there are black vertices in $\sigma \cap \tau$. If there are none, then $\sigma \cap \tau$ is the horizontal boundary of the simplices spanned by $\langle V(\sigma \cap \tau) , v_b\rangle$ for all the black vertices $v_b$ in $V(\sigma \cup \tau)$. Since we assumed that there is more than one black vertex, $I(X,\varepsilon)$ could not be a matching.

Similarly, let us note that $I(X,\varepsilon)$ cannot be a Morse matching if there exists a simplex with more than one black vertex (as there would be more than one face of this simplex with the same number of white vertices, hence more than one edge emanating from the simplex in $I(X, \varepsilon)$).  

This implies at once that $\sigma \cap \tau$ cannot contain more than one black vertex. If instead $\sigma \cap \tau$ contains exactly one black vertex, by our initial assumption (by contradiction) there must be at least another black vertex in either $\sigma$ or $\tau$, and again we would have a simplex with more than one black vertex, preventing $I(X, \varepsilon)$ from being a Morse matching.
~\\

$\Longleftarrow )$ Let us begin by showing that the condition on the colouring implies that $I(X, \varepsilon)$ is a matching. Consider a simplex $\sigma \subseteq X$; if $\sigma$ is contained in the horizontal boundary of some simplex $\sigma^\prime$, then by definition $w_\varepsilon (\sigma) = w_\varepsilon (\sigma^\prime)$. In particular $\sigma^\prime = \langle V(\sigma) , v_b\rangle$ for some black vertex $v_b$, and $\sigma$ is the only $\partial_h$-boundary of $\sigma^\prime$ (see Figure~\ref{fig:outwardflow} for an explicit example); the uniqueness of such a $\sigma$ follows  from the fact that $\sigma^\prime$ has a unique black vertex. Furthermore in this case --using the fact that $\varepsilon$ is dalmatian-- we have $w_\varepsilon (\sigma) = dim(\sigma) +1$ (\emph{i.e.}~all of its vertices are white), hence $\partial_h (\sigma) =0$, and thus $\sigma$ is not matched to any simplex of lower dimension.

If instead $\sigma$ is not the $\partial_h$ boundary of another simplex, then either $w_\varepsilon (\sigma) = dim(\sigma) +1$ or $w_\varepsilon (\sigma) = dim(\sigma)$. In the former case $\sigma$ is not matched to any simplex (\emph{i.e.}~it is a critical simplex), and in the latter $\sigma$ is only matched with its unique white face.

The last thing we need to address is the acyclicity of the matchings $I(X, \varepsilon)$. Consider two $n$-dimensional simplices $\sigma, \sigma^\prime$ sharing a codimension $1$ face $\tau$. If $\sigma$ and $\tau$ are matched by $I(X,\varepsilon)$, \emph{i.e.} the $\varepsilon$-horizontal differential of $\sigma$ is $\tau$, then (using the dalmatian hypothesis) we see that $\sigma^\prime$ must have trivial $\varepsilon$-horizontal boundary. This implies the non-existence of oriented cycles induced on $\hasse$ by $I(X, \varepsilon)$.
\end{proof}

\begin{figure}[ht]
\centering
\includegraphics[width = 5cm]{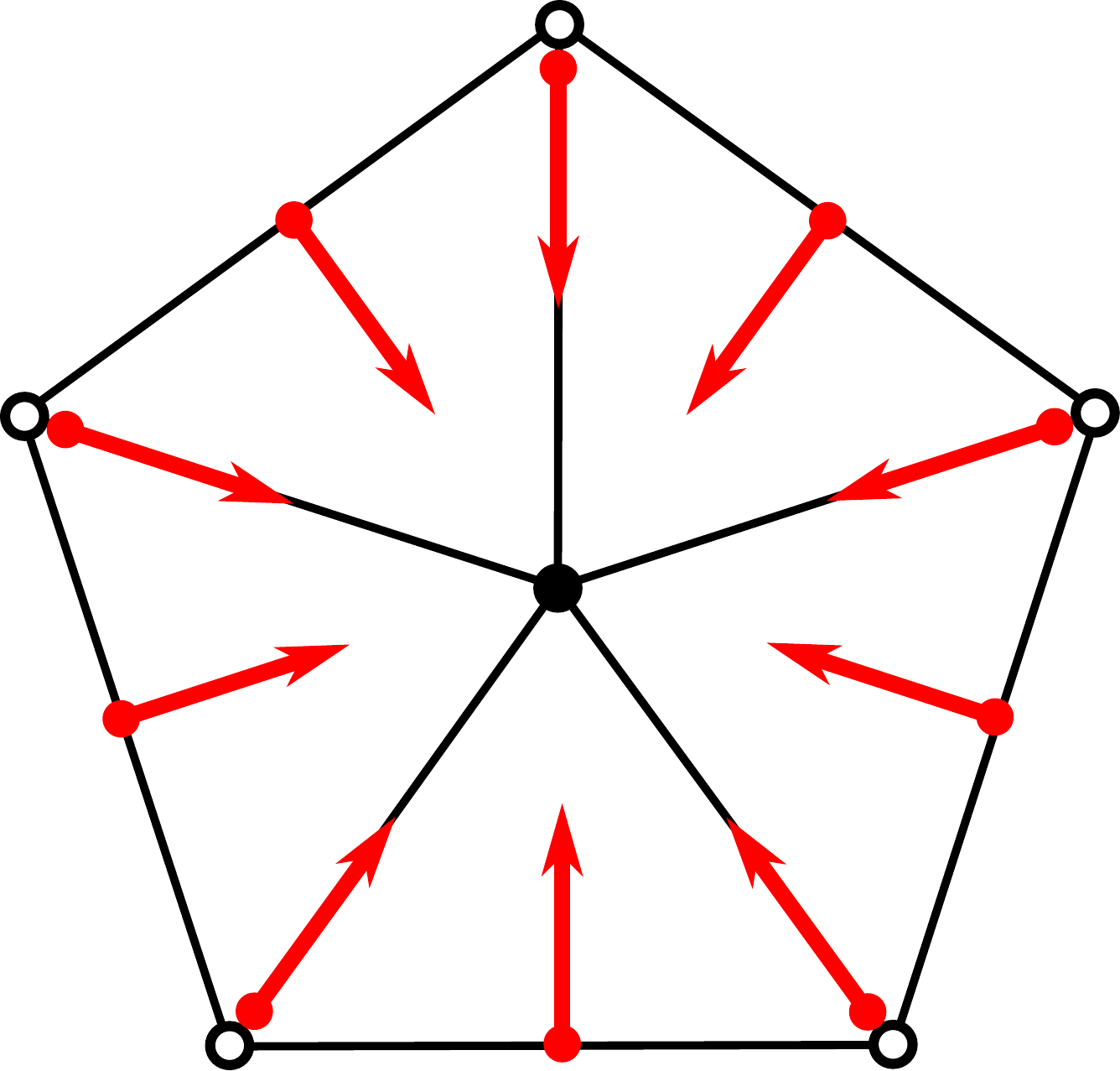}
\caption{Red arrows denote non-trivial components in the horizontal boundary, oriented backwards. The combinatorial flow induced by an elementary subgraph around the black vertex is ``inward pointing''. This localised flow cannot create oriented cycles if the colouring is dalmatian.}
\label{fig:outwardflow}
\end{figure}

It follows trivially from the definition of $I(X,\varepsilon)$ that if two simplices are matched, then they have the same $w_\varepsilon$ degree in $C_*(X, \varepsilon)$, and their dimensions differ by one. Also, let us observe that if a $n$-simplex $\sigma$ in $(X, \varepsilon)$ contains $k$ black vertices, then its horizontal differential has exactly $k$ components; equivalently, $I(X, \varepsilon)$ has $k$ edges emanating from $\sigma$.

\begin{exa}
We outline here how to determine $I(X, \varepsilon_i)$ for the elementary subgraphs of the simplices $\Delta^m$; it turns out that all of these are in fact maximal. (This holds more generally for all simplicial complexes $X$ for which there exists a vertex $v$ such that $\overline{St(v)} = X$.) 

Consider the simplex $\Delta^m$ with vertices $\{v_1, \ldots, v_{m+1}\}$; we can assume wlog that $v_1$ is the only black vertex, so $\varepsilon = \varepsilon_1$. Then we claim that the only vertex of $\mathcal{P}(\Delta^m)$ which is not matched by $I(\Delta^m, \varepsilon_1)$ is $v_1$. 

This can be seen recursively: the whole simplex $\langle v_1, \ldots, v_{m+1} \rangle$ is paired with the only codimension $1$ face not containing $v_1$; then all other faces either contain $v_1$ or do not. In the former case their unique $\varepsilon$-horizontal boundary is obtained by removing $v_1$, while in the latter case we have the dual situation (see Figures~\ref{fig:simplexcomplex} and \ref{fig:colouredcpx2simplex}), as they are matched with the simplex obtained by adding $v_1$.  
\begin{figure}[ht]
\centering
\includegraphics[width = 11cm]{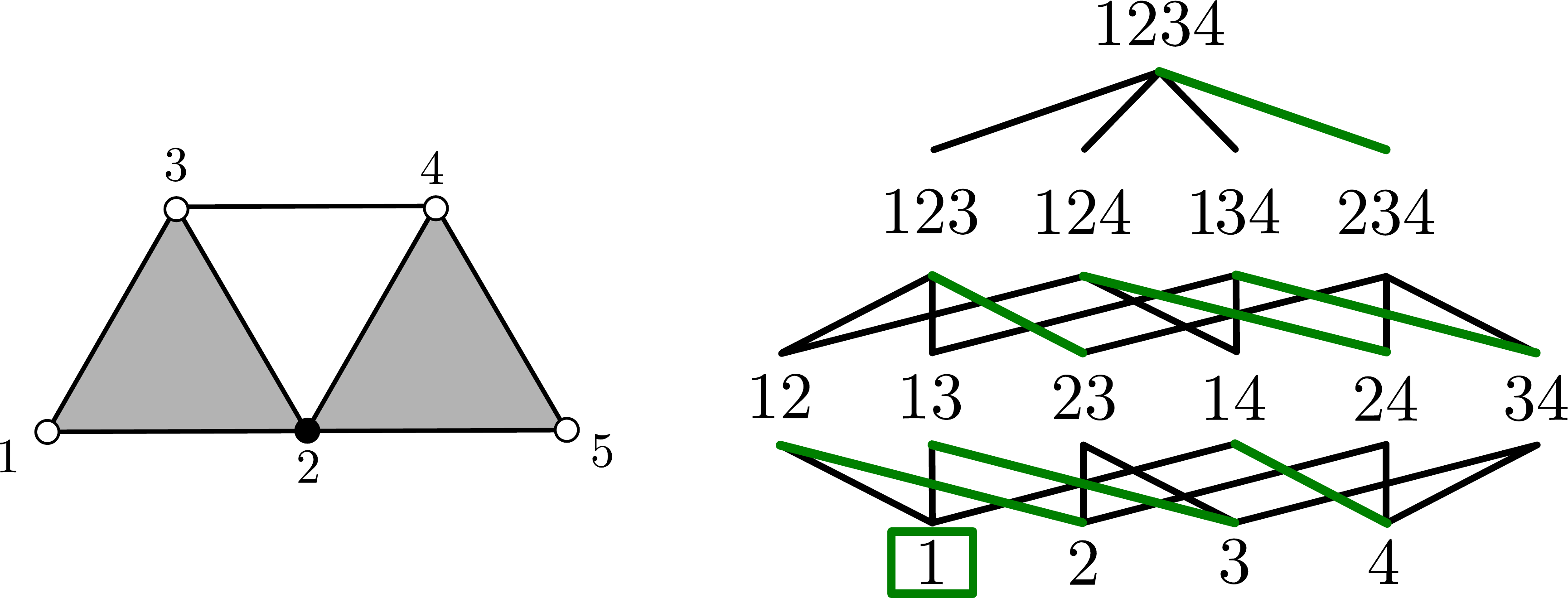}
\caption{On the left, a simplicial complex; The colouration displayed induces a maximal Morse matching on its face poset. There are exactly two critical simplices, $\langle 3,4 \rangle$ in dimension $1$ and $\langle 2 \rangle$ in dimension $0$, coherently with the fact that $X$ has the homotopy type of $S^1$. On the right, a maximal Morse matching on the face poset of the $3$-simplex, induced by colouring in black only the vertex $1$. Matched edges are in green, and the only generator of the horizontal homology (and only unmatched simplex) is encased in the green square.}
\label{fig:simplexcomplex}
\end{figure}
\end{exa}

As described in Forman's seminal paper~\cite{forman1998morse}, it is possible to compute the homology of a simplicial complex using discrete Morse matchings; in particular, there is an injection between generators of $\HH_n (X)$ and critical $n$-simplices of a discrete Morse matching on $\hasse$. The following easy consequence of Theorem~\ref{thm:dmmandsubgraphs} implies that, in the special case of Morse matchings induced by colourings, we get a bijection between these critical simplices and generators of the horizontal homology:

\begin{cor}\label{cor:criticalpoints}
If the $\varepsilon$-coloured simplicial complex $(X, \varepsilon)$ is dalmatian, then the generators of $\HH^h_n(X, \varepsilon)$ are in bijection with the critical $n$-cells of the Morse matching $I(X,\varepsilon)$.
\end{cor}
\begin{proof}
We can start by observing that, since the horizontal differentials in $C(X, \varepsilon)$ are identified with the edges of $I(X, \varepsilon)$, all critical cells of $I(X, \varepsilon)$ are automatically non-trivial in $\HH^h(X, \varepsilon)$. On the other hand, two matched simplices form an acyclic subcomplex (these are called \emph{atoms} in \cite[Ch.~11.3.2]{kozlov2007combinatorial}), and are thus trivial in $\HH^h(C, \varepsilon)$.
\end{proof}
As a consequence of Corollary~\ref{cor:criticalpoints}, we can use the horizontal homology to measure ``how far'' away a dalmatian $I(X, \varepsilon)$ is from being a perfect Morse matching in any given homological degree, by considering the difference $\mathrm{rk}(\HH^h_i(X, \varepsilon)) - \mathrm{rk}(\HH_i(X))$.

\begin{defi}\label{def:flattening}
Given a bi-graded chain complex $C = \bigoplus_{i,k \in \N} C(i,k)$, we can consider its \emph{flattening}, that is the singly graded complex $$\mathcal{F}(C) = \left( C^f_i = \bigoplus_k C(i,k),\; \partial^f =  \bigoplus_k \partial_{i,k}\right)$$
obtained by ``forgetting'' the filtration. 
\end{defi}

As an example, if $\varepsilon = (0, \ldots,0)$, then $\mathcal{F}(\HH^h_*(X, \varepsilon)) \cong C_*(X)$, as the horizontal differential is trivial (\emph{cf.}~Equation~\eqref{eqn:zeroepsilon}).

\begin{rmk}
If $\varepsilon$ is dalmatian, Corollary~\ref{cor:criticalpoints} together with a result of Kozlov \cite[Thm.~11.24]{kozlov2007combinatorial} imply the existence of a further differential $\partial_{\mathcal{M}}$ on the flattening  $\mathcal{F}(\HH^h(X, \varepsilon))$ such that $$\HH_*\left( \mathcal{F}(\HH^h(X, \varepsilon)), \partial_{\mathcal{M}} \right) \cong \HH_*(X).$$
\end{rmk}

It follows immediately from Theorem~\ref{thm:dmmandsubgraphs} that all elementary subgraphs are in fact Morse matchings on $\hasse$; we will now show that in some sense they provide the ``building blocks'' for all the homologies $\HH^h(X, \varepsilon)$ (see Figure~\ref{fig:simplexpartition}). This yields a unique non-trivial decomposition of the simplicial differential in components induced by discrete Morse matchings. Recall that $\varepsilon_i$ denotes the colouring in which $v_i$ is the unique black vertex.

\begin{thm}\label{thm:decomposition}
Let $(X,\varepsilon)$ be a finite bi-coloured simplicial complex with vertices $\{v_1, \ldots, v_m\}$. Then
$$I(X, \varepsilon) = \bigsqcup_{\{i \,|\,\varepsilon(i) =1\}} I(X, \varepsilon_i). $$
Here the disjoint union indicates that the edges are distinct, but might intersect in their endpoints.
\end{thm}
\begin{proof}
We begin by showing that $E(I(X, \varepsilon_i)) \cap E(I(X,\varepsilon_j)) = \emptyset$ if $i\neq j$. This follows from the fact that if $e$ is an edge in $I(X,\varepsilon_i)$, connecting the simplex $\sigma$ to the simplex $\tau$, then $V(\sigma) = V(\tau) \cup \{v_i\}$. In particular, $e$ cannot belong to  $I(X, \varepsilon_j)$ for any $j \neq i$.

We can then argue that if an edge $e$ in the face poset belongs to $E(I(X, \varepsilon))$, then it belongs to a unique elementary subgraph $I(X, \varepsilon_i)$ for some $i \in 1, \ldots, |V(X)|$. Indeed, assume that the edge $e$ connects $\sigma$ to $\tau$; this implies that $w_\varepsilon (\sigma) = w_\varepsilon (\tau)$, and thus $\sigma$ has exactly one more black vertex $v_i$ than $\tau$, thus $e \in I(X, \varepsilon_i)$.

To conclude, we only need to show that if an edge $e \in E(I(X,\varepsilon))$ belongs to some $I(X,\varepsilon_i)$, then  all other edges in $I(X,\varepsilon_i)$ are contained in $I(X, \varepsilon)$ as well. This follows immediately from the previous analysis: if an edge belongs to $I(X, \varepsilon_i)$, then the vertex $v_i$ must be black in $\varepsilon$. Moreover, for each simplex in $X$ containing $v_i$, the sub-simplex obtained by removing $v_i$ has the same $w_\varepsilon$ weight, hence the edge belongs to both $I(X,\varepsilon)$ and $I(X, \varepsilon_i)$.
\end{proof}
\begin{figure}[ht]
\centering
\includegraphics[width = 8cm]{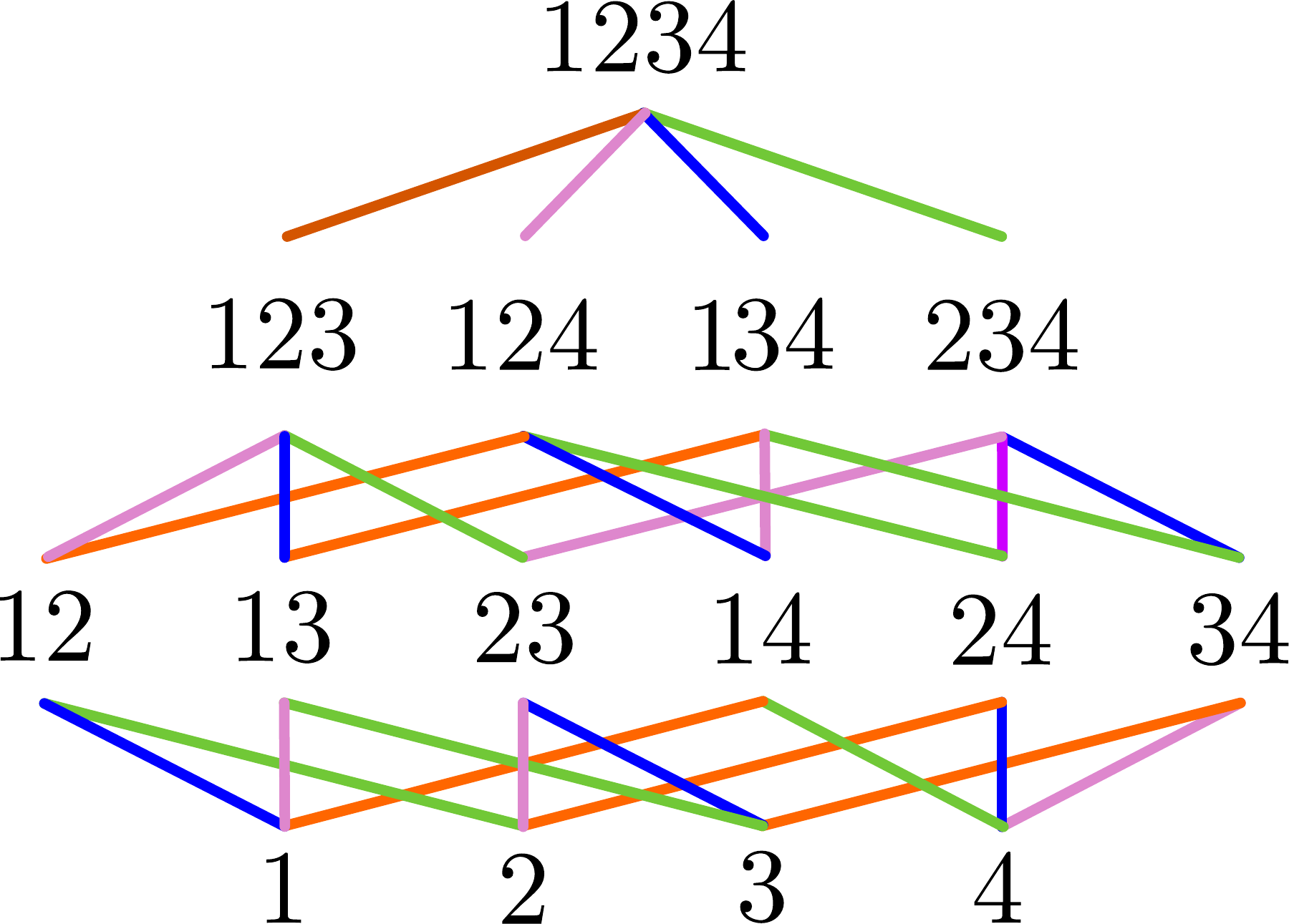}
\caption{The partition of the edges in $\mathcal{P}(\Delta^3)$ for $\varepsilon = (\textcolor{Green}1\color{black},\color{blue}1\color{black},\textcolor{pink}1\color{black},\color{orange}1\color{black})$. The first, second, third and fourth black vertices induce respectively the green, blue, violet and orange colouring of the edges. The union of these edges is the horizontal differential in the flattened chain complex $\mathcal{F}(C(\Delta^3, (1,1,1,1))) \cong C(\Delta^3)$.}
\label{fig:simplexpartition}
\end{figure}

As a noteworthy special case, if we choose the all black colouring $\varepsilon = (1, \ldots, 1)$, we can uniquely partition the simplicial differential on $X$ in $|V(X)|$ sub-differentials, each of which induces a Morse matching on $\hasse$, as shown in Figure~\ref{fig:simplexpartition} for $X = \Delta^3$. 

This could be used to efficiently compute all the horizontal homologies of $X$ as follows (see also  Section~\ref{sec:graphdiss}):
first compute the $m$ horizontal homologies in the elementary colours, then the matrices representing the horizontal differentials for a generic colour can be easily obtained by applying Theorem~\ref{thm:decomposition}.

Similarly, the relation with discrete Morse matchings outlined in this section allows us to easily compute the horizontal homology of $(X,\varepsilon)$ for dalmatian colourings; in particular, generators of $\HH^h(X, \varepsilon)$ are in bijection with  black vertices, together with simplices of maximal filtration degree (\emph{i.e.} such that $w_\varepsilon (\sigma) = dim(\sigma) +1$) that do not belong to the link of the black vertices. In what follows, the notation $\F\langle \sigma\rangle_{(a,b)}$ will denote a copy of $\F$ generated by the simplex $\sigma$ in bi-degree $(a,b)$.

\begin{prop}\label{prop:dalmatianhor}
Let $(X,\varepsilon)$ be a dalmatian simplicial complex; call $b(\varepsilon)$ the set of indices corresponding to black vertices. Then 
\begin{equation}\label{eqn:horizz}
\HH^h(X,\varepsilon) \cong \bigoplus_{r \in b(\varepsilon)} \F\langle v_r\rangle_{(0,0)}
\bigoplus_{\substack{\sigma \subseteq X \\ \sigma \nsubseteq \cup_{r \in b(\varepsilon)} \overline{St(v_r)}}} 
\F \langle \sigma \rangle_{(dim(\sigma), dim(\sigma)+1)}.
\end{equation}
\end{prop}
\begin{proof}
Since by Theorem~\ref{thm:dmmandsubgraphs} we know that $I(X, \varepsilon)$ is a Morse matching, the complex $\left( C(X, \varepsilon), \partial_h \right)$ can be decomposed in pairs of matched generators (which constitute an acyclic pair), or unmatched generators, which are non-trivial in $\HH^h(X, \varepsilon)$. We just need to identify these latter generators with the two summands in Equation~\eqref{eqn:horizz}. 

As the colouring is dalmatian, the black vertices in $(X, \varepsilon)$ are the only simplices in $w_\varepsilon$ degree $0$. This implies they are not matched in $I(X, \varepsilon)$, and they give the left summand. 

For the rightmost term on the RHS of Equation~\eqref{eqn:horizz}, note that all the simplices in the set $\{\sigma \subseteq X\,|\, \sigma \nsubseteq \cup_{r \in b(\varepsilon)} \overline{St(v_r)}\}$ are not matched. These are the only unmatched ones, since each simplex with non-trivial differential contain a unique black vertex; hence the only matched simplices either contain a black vertex or are the face opposite to a black vertex in a simplex, and thus in particular lay in the closed star of said vertex (as shown in Figure~\ref{fig:outwardflow}).
\end{proof}

The ``local'' acyclic matchings induced by elementary colourings (as in Figure~\ref{fig:simplexcomplex}) were previously considered in a different context in \cite{singh2020higher}, where they are used to recursively construct global acyclic matchings on the whole complex. We review and slightly extend their construction using our notations. Start from a dalmatian colouring $\varepsilon$ on $X$; remove all pairs of matched simplices (note that this leaves out the critical $0$-cells given by the black vertices of the colouring). In particular, all black vertices are critical $0$-simplices. Then iterate the same process on what is left; more precisely, choose a sequence $\{D_r\}_{r = 0, \ldots, \mu(X)}$ of dalmatian colourings for $X$, with the properties $$V\left(\bigcup_{r = 0}^{\mu(X)} \bigcup_{v \in D_r} \overline{St(v)} \right) = V(X),$$ and $$D_{p+1} \subseteq V\left(X \setminus \bigcup_{r = 1}^p \bigcup_{v \in D_r} \overline{St(v)}\right).$$
In particular, all vertices in $X$ are either coloured in black in some $D_r$, or are at distance $1$ from some coloured vertex. 

\begin{figure}[ht]
\centering
\includegraphics[width=10cm]{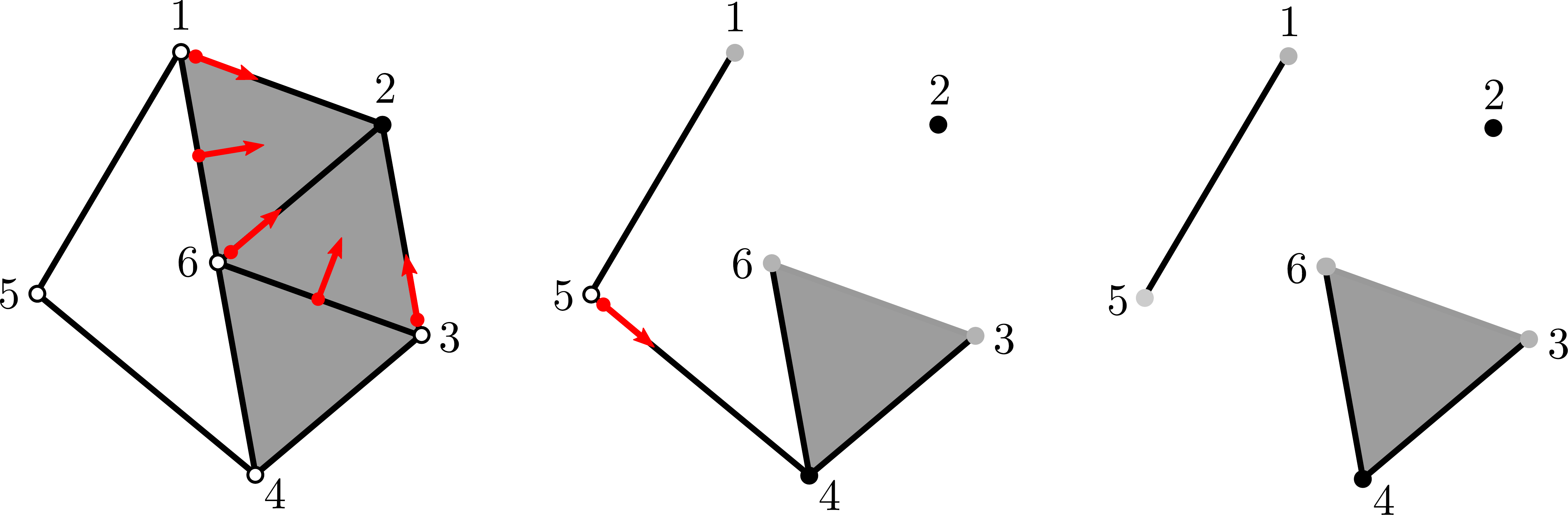}
\caption{The sequence of spaces and matchings $(X_p,M_p)$ for the complex $X$ on the left. We start by colouring the vertex $2$ in black, and in the second step we colour $4$. Grey vertices are to be regarded as ``removed'', and red arrows denote matched pairs. The end result consists of $2$ critical $0$-cells (the two black vertices), $3$ critical $1$-cells (the edges $\langle15\rangle,\langle46\rangle$ and $\langle34\rangle$) and the critical $2$-cell $\langle346\rangle$.}
\label{fig:iterateddalmatian}
\end{figure}

Then define the sequence of matchings $M_p$ as the set of pairs of matched simplices with respect to the colouring $D_p$ on $X_p$. Note that $M_p$ can be identified with the subgraph $I(X,D_p) \subseteq \hasse$. 
In turn, define the sequence of spaces $X_0 = X$ and  $X_{p+1} = X_p \setminus \{\sigma  \in X_p \,|\, \sigma \in \pi, \text{ for some } \pi \in M_p\}$ (see Figure~\ref{fig:iterateddalmatian}). In other words, $X_{p+1}$ is composed by the collection of critical cells in $X_p$ with respect to the matching $M_p$. Note that $X_p$ is generally not a simplicial complex. 

\begin{prop}[\cite{singh2020higher}, Prop.~$2.4$]\label{prop:iterateddalmatian}
The matching $\displaystyle \bigsqcup_{p = 0}^{\mu(X)} M_p$ is a Morse matching on $X$.
\end{prop}
The only difference between the proofs of Proposition~\ref{prop:iterateddalmatian} and~\cite[Prop.~$2.4$]{singh2020higher} is that here we are allowing more than one vertex at each step in the iteration. Since each $M_p$ is defined by a dalmatian colouring, all steps can be performed without changes.

As shown in Figure~\ref{fig:dalmatching}, for  simplicial complexes with a small diameter, the acyclic matching determined by an elementary colouring is far from being perfect; this will actually come as good news when combined with  Theorem~\ref{thm:0dim} (see also Remark~\ref{rmk:othersubdivisions}).
\begin{figure}[ht]
\centering
\includegraphics[width=9cm]{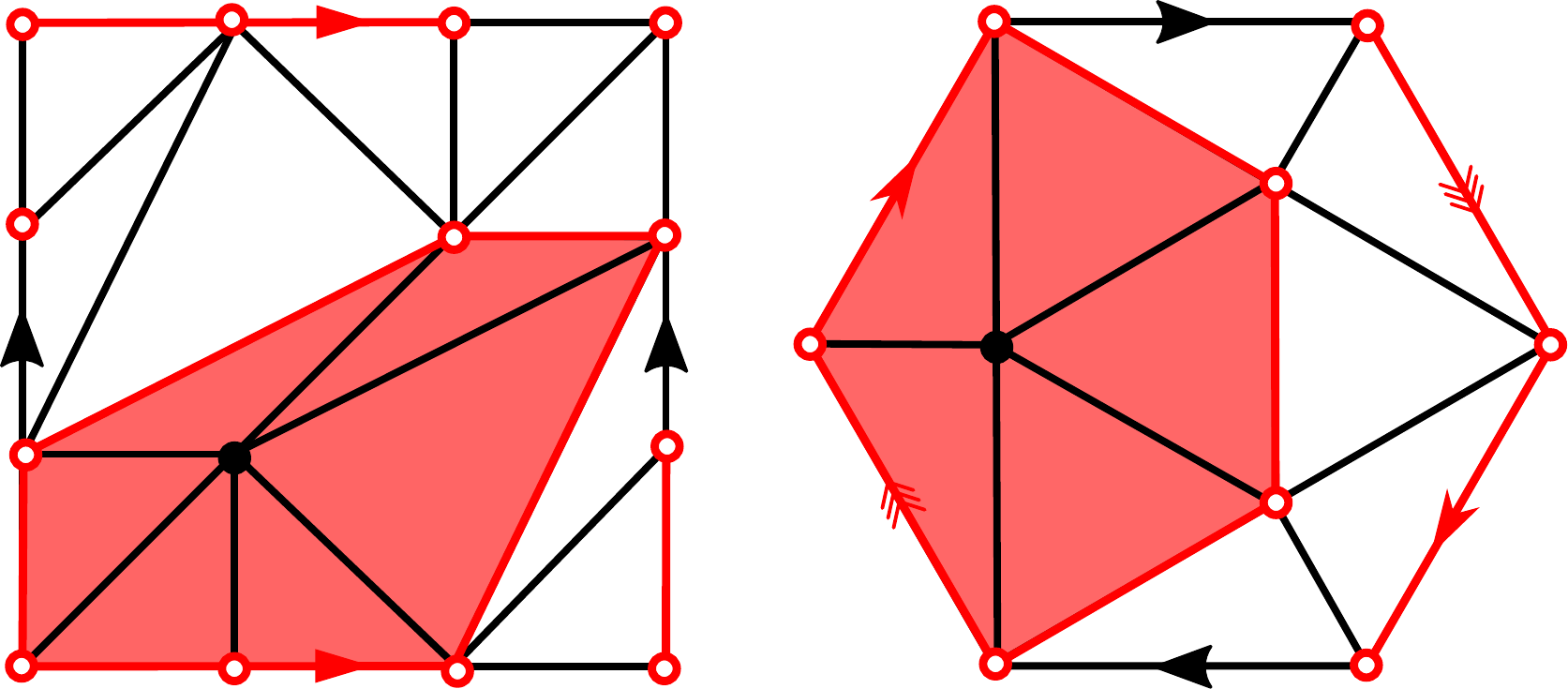}
\caption{From left to right, a minimal triangulation for the $2$-torus and for $\R\mathbb{P}^2$. In both cases any elementary colouring ``covers'' all vertices; for the specific choice of black vertices above we get $(1,9,8)$ and $(1,5,5)$ critical cells in dimensions $0,1,2$ respectively.}
\label{fig:dalmatching}
\end{figure}

\section{Bi-coloured complexes from graphs}\label{sec:graphs}

In this section we provide some examples from graph theory where $\varepsilon$-coloured simplicial complexes arise naturally, and that motivated our initial studies on the matter.

The goal here is to show that, in the specific case of plane graphs, the horizontal homology of matching complexes exhibits a rich structure that is tightly connected to the properties of the graph we started with. More precisely Theorem~\ref{thm:horizhomoftait} will show the existence of a decomposition of the horizontal homology of the matching complex of certain graphs in terms of smaller pieces determined by the combinatorics of matchings.

We will always assume that our graphs are connected and have a finite number of vertices. Given a plane graph $G$,  we can associate to it a new plane graph $\Gamma(G)$, usually called the \emph{overlaid Tait graph of $G$}. This is obtained by overlaying $G$ with its unique plane dual $G^*$, as shown in Figure~\ref{fig:planeduals}. The vertex set of $\Gamma(G)$ is naturally tri-partite:
$$V(\Gamma(G)) = V(G) \sqcup V(G^*) \sqcup (E(G) \cap E(G^*)),$$ where the last set of vertices is given by the unique intersection between each edge in $G$ and its dual in $G^*$.

Overlaid Tait graphs are well known both in graph theory (\emph{cf}.~the construction due to Kenyon, Propp and Wilson~\cite{kenyon1999trees}) and in knot theory (\emph{cf.}~\cite{kauffman2006formal}). 

In the latter subject, the plane graph $G$ is usually obtained by chequerboard colouring the regions in the complement of a projection of a link; black regions are then taken as vertices, and double points in the projection are the edges. For this reason, $G$ is usually referred to as the black graph associated to the link's projection.
As all plane graphs arise from such a construction for some link projection, we can use the two points of view interchangeably.

\begin{figure}[ht]
\centering
\includegraphics[width = 11cm]{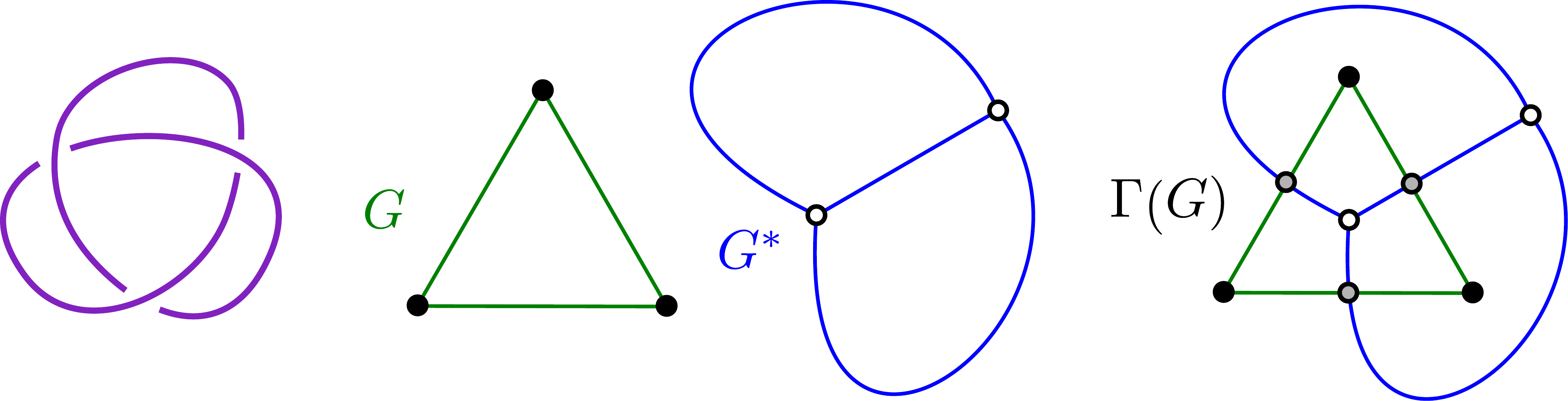}
\caption{The overlaid Tait graph for the loop $L_3$ of length three. This arises as the black graph of the minimal projection of the trefoil knot shown on the left. The grey vertices in $\Gamma(L_3)$ are given by the unique intersection between pairs of dual edges.} 
\label{fig:planeduals}
\end{figure}

Since by construction all edges in $\Gamma(G)$ have exactly one of their endpoints belonging to $V(G) \cup V(G^*)$, we  colour an edge in $\Gamma(G)$ black or white depending on whether one of its vertices belongs to $V(G)$ or $V(G^*)$ respectively. We refer to this choice of edge colours of $\Gamma(G)$ as $\varepsilon_G$.

So, for a given plane graph $G$, using this edge colouring we can construct $\left( \Mat (\Gamma(G)), \varepsilon_G \right)$ the $\varepsilon_G$-coloured  matching complex of $\Gamma(G)$. Vertices of $\Mat (\Gamma(G))$ are coloured according to the colour of the corresponding edge in $(\Gamma(G), \varepsilon_G)$. As shown in Figure~\ref{fig:planeduals}, edges in the barycentric subdivision of $G$ (which can be identified with half-edges in $G$) correspond to edges in $\Gamma(G)$ having exactly one black and one grey vertex. 
The following result tells us that  the homologies of the matching complex of the first barycentric subdivision of $G$, denoted by $B(G)$, and of $\Gamma(G)$ sit at the ``opposite sides'' of a filtered chain complex:

\begin{lem}\label{lem:filtrationtait}
Consider the $\varepsilon_G$-coloured chain complex $\left( C_*(\Mat (\Gamma(G)), \varepsilon_G), \partial \right)$, with the filtration given by $w_{\varepsilon_G}$; we have an isomorphism of chain complexes 
\begin{equation}\label{eqn:bottomcase}
\left( C_*(\Mat (\Gamma(G)),\varepsilon_G, \le0), \partial \right) \cong \left(  C_*(\Mat (B(G))), \partial \right),
\end{equation}
where on the right we have the chain complex of the simplicial homology for the matching complex of $B(G)$.

Moreover, for $k \ge |E(G)|$, then the flattening of $C_*(\Mat (\Gamma(G)), \varepsilon_G, \le k)$ is isomorphic to the chain complex whose homology is  $ \HH_*(\Mat (\Gamma(G)))$. 
In particular $$\mathcal{H}_*(\Gamma(G),\varepsilon_G, 0) \cong \HH_*(\Mat (B(G))),$$ and $$\mathcal{F}\left( \mathcal{H}_*(\Gamma(G),\varepsilon_G, |E(G)|) \right) \cong \HH_*(\Mat (\Gamma(G))).$$
\end{lem}
\begin{proof}

Let us start with the case where $k = 0$; elements in $C_*(\Mat (\Gamma(G)),\varepsilon_G, \le 0)$ are formal sums of matchings in $\Gamma(G)$ involving only black edges. But since by construction black edges in $\Gamma(G)$ are the edges in $B(G)$ (\emph{i.e.}~the half-edges in $G$), we get a bijection between the generators of the two complexes in Equation~\eqref{eqn:bottomcase}. The differentials on both sides also agree:  all simplices generating the complex on the left-hand side of Equation~\eqref{eqn:bottomcase} are coloured in black, hence in filtration degree $0$ the simplicial differential reduces to the horizontal one. 

The isomorphism between $\mathcal{F}\left(\mathcal{H}_*(\Mat(\Gamma(G)),\varepsilon_G, |E(G)|)\right) \cong \HH_*(\Mat (\Gamma(G)))$ follows from the definition, after noting that if $k \ge |E(G)|$ we have an equality of complexes $C_*(\Mat(\Gamma(G)), \varepsilon_G, \le k) = C_*(\Mat(\Gamma(G)))$. This is because $dim(\Mat(\Gamma(G))) \le |E(G)|$, which can be easily seen by noting that $\Gamma(G)$ is a bipartite graph, with respect to the decomposition   $V(\Gamma(G)) = (V(G) \sqcup V(G^*)) \sqcup (E(G) \cap E(G^*))$.  Furthermore, using simple Euler characteristic considerations we have $|V(G)| + |V(G^*)| = |E(G)|+2$, and the number of edges in a matching on $\Gamma(G)$ is less or equal to $|E(G)|$.
\end{proof}

While it is not generally easy to compute the homology groups of $C_* (\Gamma(G), \varepsilon_G, \le k)$ interpolating between $\HH_*(\Mat (B(G)))$ and $\HH_*(\Mat (\Gamma(G)))$, it is possible to explicitly determine the homology of the associated graded object, or in other words, the horizontal homology of $\Mat (\Gamma(G))$, in terms of the homology of simpler chain complexes. For a matching $m \in Match(B(G^*))$, define $G(m)$ as the subgraph of $G$ obtained by removing all edges with duals whose barycentres are matched by $m$. One example is displayed in Figure~\ref{fig:complementarygraph}.

The content of the next theorem is that, in order to compute $\HH^h_*(\Mat (\Gamma(G)), \varepsilon_G, k)$, we can compute the simplicial homology of the matching complex of certain subgraphs of $B(G)$; to put it another way, the horizontal homology of the coloured complex $(\Mat (\Gamma(G)), \varepsilon_G)$ contains the information about all matching complexes of the barycentric subdivision of $G$. This result is an analogue of \cite[Cor.~13]{celoria2020filtered}, where a different filtration on $\Mat(\Gamma(G))$ was considered.
Note that the colouring $\varepsilon_G$ is not dalmatian, hence the next result does not overlap with Proposition~\ref{prop:dalmatianhor}. Recall that $\F_{(a,b)}$ denotes a copy of the field $\F$ generated by simplices of dimension $a$ and $\varepsilon$-degree $b$.

\begin{thm}\label{thm:horizhomoftait}
Consider the horizontal homology of the matching complex of $\Gamma(G)$, with the colouring $\varepsilon_G$ inherited from $G$. Then there is a graded isomorphism

\begin{equation}\label{eqn:horizgraphs}
\HH^h_*(\Mat (\Gamma(G)), \varepsilon_G, k) \cong 
\bigoplus_{\substack{m \in Match(B(G^*)) \\ G(m) \neq \emptyset\\|m| = k}} \widetilde{\HH}(\Mat (B(G(m))))  \bigoplus_{\substack{m \in Match(B(G^*)) \\ G(m) = \emptyset\\|m| = k}} \F_{(k-1,k)}
\end{equation}

for $k >0$. When $k=0$ instead  $$\HH^h_*(\Mat (\Gamma(G)), \varepsilon_G, 0) \cong \HH_*(\Mat (B(G))).$$
\end{thm}
\begin{proof}
We start by noting that the case with $k = 0$ has already been covered in Lemma~\ref{lem:filtrationtait}:  the only differentials in the complex   $C_*(\Gamma(G),\varepsilon_G, \le 0)$ are horizontal, hence its homology coincides with $\HH^h_*(\Mat (\Gamma(G)), \varepsilon_G, 0)$ in this filtration degree.

Now consider a non-trivial matching $m \in Match(B(G^*))$ such that $|m| = w_{\varepsilon_G}(m) = k$;  then, as shown in Figure~\ref{fig:complementarygraph}, if a half-edge in $G^*$ belongs to $m$, the two half-edges composing the dual edge in $G$ cannot be part of any matching of $\Gamma(G)$. Hence if we want to extend $m$ to a matching $m^\prime \in Match(\Gamma(G))$ such that $w_{\varepsilon_G}(m^\prime) = k$ we can only use half edges of $B(G(m))$. 

In particular, if $G(m) = \emptyset$, the matching $m$ on $B(G^*)$, seen as a simplex in $\Mat (B(G^*))$ has only white vertices; further, it cannot be the horizontal boundary of a larger simplex, since this would contradict the assumption $G(m) = \emptyset$. 
Therefore, matchings $m$ such that $G(m) = \emptyset$, give  direct summands in $H_*^h (\Mat (\Gamma(G)), \varepsilon_G, k)$. Each of these summands corresponds to a copy of $\F$ in the rightmost summand in Equation~\eqref{eqn:horizgraphs}.
\begin{figure}[ht]
\centering
\includegraphics[width = 8cm]{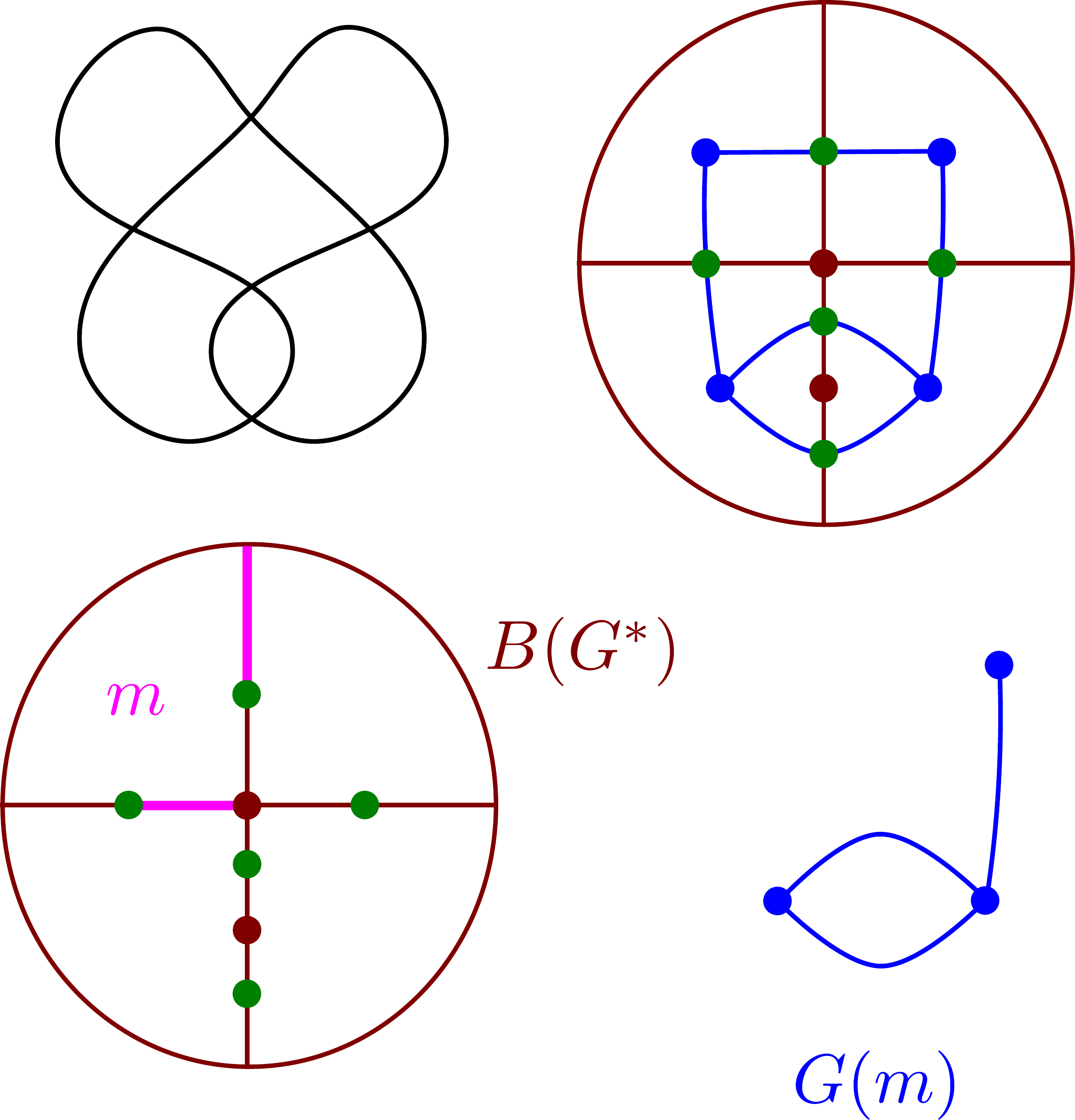}
\caption{On the top, a projection of the $5_2$ knot, followed by its overlaid Tait graph. Vertices coming from the chequerboard colouring are blue\slash red for vertices in $G$ and $G^*$ respectively, and crossings are green. The vertex in $G^*$ corresponding to the external region has been placed ``at $\infty$'' in $S^2$. In the lower part, a pink matching $m$ on the first barycentric subdivision of $G^*$ (seen as a subgraph of $\Gamma(G)$), and the complementary subgraph $G(m) \subseteq G$. Any half-edge matched in pink in $B(G^*)$ precludes the chance of matching its barycentre in $B(G)$.}
\label{fig:complementarygraph}
\end{figure}
Let us now conclude by considering the case where $m$ is such that $k >0$ and $G(m) \neq \emptyset$; if we regard $m$ as being a sub-matching of $m^\prime \in Match(\Gamma(G))$, then the associated simplex $\sigma_{m^\prime}$ in $\Mat (\Gamma(G))$ will have $w(\sigma_{m^\prime}) = k$ and $dim(\sigma_{m^\prime}) = |m^\prime| - 1$. 

Now, it follows from the definition of $\partial_h$ that the collection of the simplices of the form $\sigma_{m^\prime}$ for $m^\prime \supset m$ is a sub chain complex $C(m)$ of $C (\Mat (\Gamma(G)), \varepsilon_G,k)$. 

The homology of $\widetilde{\HH}_*(\Mat (B(G(m))))$  can be then identified with the homology of $C(m)$; at the level of chains, we obtain a bi-graded bijection by sending $m^\prime = \langle m, m''\rangle \in C(m)$ (with $m'' \in Match(B(G(m)))$ and $|m''|>0$) to $m''$, and $\langle m\rangle$ to the terminal object in the reduced complex for $\Mat (B(G(m)))$.\\
Therefore we obtain the reduced homology as required, since under this identification the differentials in the two complexes correspond bijectively.
\end{proof}

Of course, it would be very interesting to further relate properties of $\HH^h(\Mat(\Gamma(G)))$ and $G$ (and equivalently to link projections).

\section{Horizontal homology and graph dissimilarity}\label{sec:graphdiss}

Here we use the horizontal homology to define a \emph{dissimilarity} measure between graphs. A dissimilarity measure is a function that, given two graphs as input, outputs a non-negative number that somehow quantifies ``how much'' the graphs are different. Extremely simple examples include the absolute value of the difference between the number of vertices or edges, connectedness, and the number of simple cycles. 
There are also several existing sophisticated dissimilarity measures (see \emph{e.g.}~\cite{schieber2017quantification}, \cite{chartrand1998graph} and \cite{li2018efficient})  ranging from topological and combinatorial to probabilistic. 

In this section we use certain collections of  horizontal homologies of graphs to define a dissimilarity measure $\Delta$, and compute it on several examples. We  conjecture that $\Delta$ is optimal, meaning that it vanishes only for isomorphic graphs (see also Section~\ref{sec:computations} for a related theme).

This is somewhat surprising, as this homology theory arises as a mild extension of simplicial homology, which can only discern a graph's homotopy type. But since, as we'll see in a moment, $\HH^h$ simultaneously encodes both combinatorial and topological features, it can be effectively used to distinguish non-isomorphic graphs.
\begin{conj}\label{conj:dissimilarity}
The dissimilarity is an optimal measure on graphs, that is $\Delta(G_1,G_2) = 0$ if and only if $G_1 \cong G_2$.
\end{conj}

We will see later on that this conjecture is equivalent to asking that a related invariant is a complete graph invariant (Conjecture~\ref{conj:completegraph}). We will also discuss computational aspects and the evidence that  supports it.

Moreover, it turns out that this dissimilarity measure can be high even for connected regular graphs and graphs with the same Laplacian or adjacency spectrum. In what follows we'll restrict to simple and connected graphs.

Given three connected graphs with the same number of vertices, it is not hard to prove (see Proposition~\ref{prop:triangleineq}) that $\Delta$ satisfies the triangular inequality, hence in particular $\Delta$ is a \emph{pseudometric} on the space of graphs. Therefore, the previous conjecture is equivalent to proving that $\Delta$ is in fact a metric on the space of graphs.

\begin{rmk}
Before proceeding, we must note that all the theory developed in this section can be applied just as well to general simplicial complexes, rather than only graphs. We will further examine this more general dissimilarity measure for simplicial complexes in a further paper.
\end{rmk}

We can now start to describe in more detail the structure of the horizontal homology for graphs, and the information that is encoded in it.

Let us begin by noting that the horizontal homology of a $\varepsilon$-coloured simple graph $(G, \varepsilon)$ has at most filtration degree $2$. Moreover the horizontal complexes and homologies admit a simple description; in filtration degree $2$ we see that complex and homology coincide, and are generated by all the edges in $G$ whose endpoints are both white. 

In filtration degree $1$ the complex is generated by edges with one endpoint of each colour and white vertices. In this case, the homology is 
\begin{equation}\label{eqn:horgraph}
\HH_*^h(G, \varepsilon, 1) \cong \bigoplus_{v \in V_{ww}} \F_{(0,1)}  \bigoplus_{v \in V_{wb}} \F^{bdeg(v)-1}_{(1,1)}.
\end{equation}
Here $V_{ww}$ and $V_{wb}$ denote the white vertices in $(G, \varepsilon)$ such that all of their neighbours are respectively either all white or not, and $bdeg(v)$ is the number of black vertices at distance one from a given white vertex in $V_{wb}$. We adopt the convention that $\F^0 = 0$.
The leftmost direct summand in Equation~\eqref{eqn:horgraph} is generated by all white vertices that are not horizontal boundaries of any $1$-simplex, which can be seen to coincide with the elements of $V_{ww}$. The rightmost summand instead is generated by $1$-simplices with exactly one white vertex; for each element $v \in V_{wb}$, there are exactly $bdeg(v)$ such edges. Each edge has $v$ as its only horizontal boundary, hence in the complex their formal sum gives $\F^{bdeg(v)-1}_{(1,1)}$.

Finally, in filtration degree $0$ the horizontal homology is isomorphic to the simplicial homology of the black subgraph $Bl(G,\varepsilon)$, \emph{i.e.}~the possibly disconnected subgraph of $G$ composed only by black vertices and the edges between them. \\

It is immediate to note that Equation~\eqref{eqn:zeroepsilon} implies that the rank of $\HH^h(G, (0, \ldots,0))$ in bi-degree $(0,1)$ is $m$, the number of vertices, while in bi-degree $(1,2)$ it coincides with the number of edges in $G$ (hence we recover the Euler characteristic of $G$). \\
If instead $\varepsilon = \varepsilon_i$ for some $i = 0,\ldots, m$, then $\mathrm{rk}(\HH^h(G, \varepsilon_i,2))$ is the number of edges not incident to $v_i$; in particular, the difference with the rank for the all-white colouring in the same bi-degree is the degree of $v_i$.
Thus the degree sequence of $G$ can be extracted from the horizontal homologies with one black vertex. 

As a final example, let us analyse what can we infer on $G$ from the computation of $\HH^h(G, \varepsilon)$ for $|\varepsilon| = 2$. Call $\varepsilon_{ij}$ the colouring with only $v_i$ and $v_j$ black, as in Figure~\ref{fig:complexesgraphs}.
Firstly, observe that we can detect if the two black vertices are adjacent in $G$ or not; more precisely $\mathrm{rk}(\HH^h(G, \varepsilon_{ij},0)) = 1$ if and only if $\langle v_i,v_j\rangle \in E(G)$. 

\begin{figure}
\centering
\includegraphics[width=14cm]{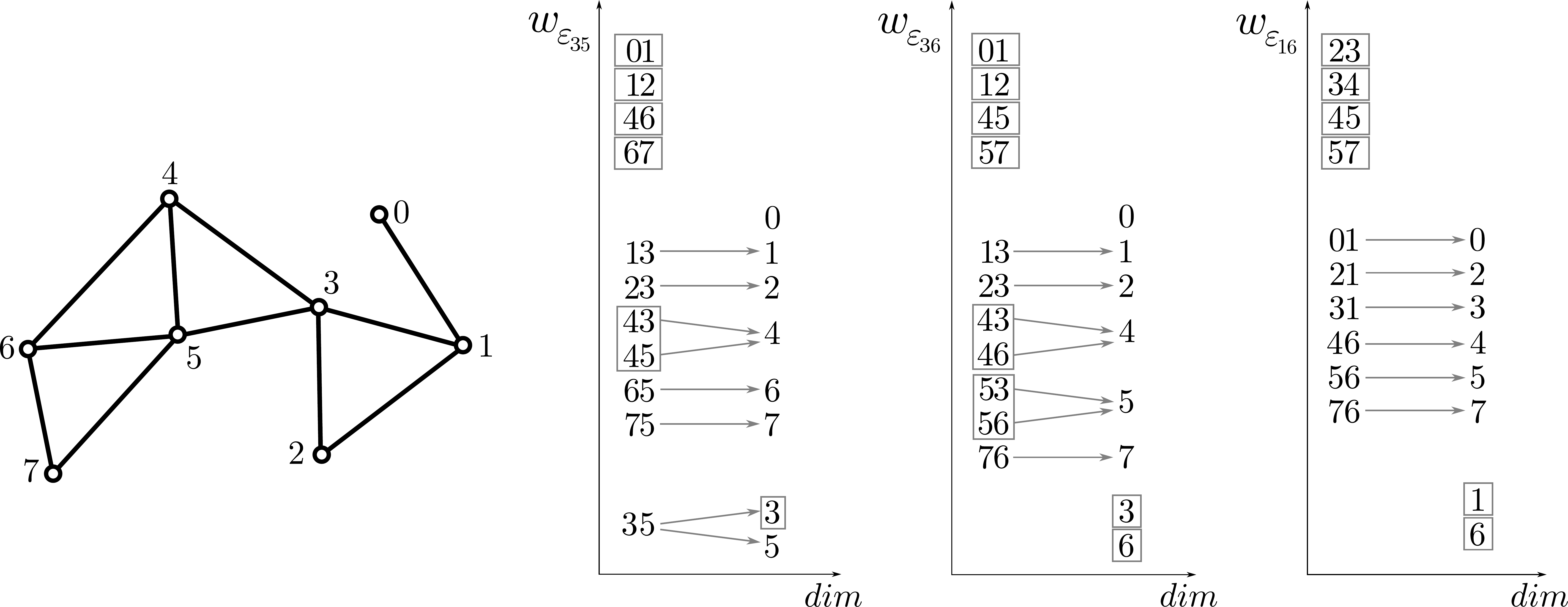}
\caption{From left to right: the graph $G$ considered, and its horizontal complexes for $\varepsilon_{35}, \varepsilon_{36}$ and $\varepsilon_{16}$ respectively. Highlighted elements denote non-trivial generators in homology; when two generators are in the same box, their sum is non-trivial in homology. The ranks of the homology in the lowest filtration degree can tell whether the two vertices are adjacent, while the rank in degree $1$ gives the number of vertices at distance one from the two black ones.}
\label{fig:complexesgraphs}
\end{figure}

As shown in Figure~\ref{fig:complexesgraphs}, for adjacent vertices, the ranks in filtration degrees two and one equal $|E(G)| - |E(St(v_i))| - |E(St(v_j))| $ and the number of vertices at distance one from both $v_i$ and $v_j$ respectively.
The same holds for a pair of vertices in $G$ that are at distance $2$. However in this case we have $\mathrm{rk}(\HH^h(G, \varepsilon_{ij},0)) = 2$. 

If instead $v_i$ and $v_j$ are at distance greater than $2$ (in particular implying that $\varepsilon_{ij}$ is dalmatian) the homology in filtration degree one is trivial, and has again rank equal to $|E(G)| - |E(St(v_i))|- |E(St(v_j))|$ in filtration degree $2$.\\

We are ready to define the first invariant of graphs obtained by collecting the ranks of specific horizontal homologies, which will then lead us to the definition of the dissimilarity measure $\Delta$ mentioned in the introduction.

\begin{defi}
For a connected and simple $\varepsilon$-coloured graph $(G, \varepsilon)$, we can encode the ranks of its horizontal homology groups with an ordered collection of $4$-tuples; more precisely call $$\widehat{\Theta}(G,j) = \left( \left( \varepsilon, i, k,r\right)_{|\varepsilon| = j} \right)\, \text{ if }\,\mathrm{rk}(\HH^h_i(G,\varepsilon, k)) = r>0.$$ Similarly, call  $\Theta (G,j)$ the ordered collection of $4$-tuples $\left(\left( j, i, k,r\right)_{|\varepsilon| = j} \right)$.
In the latter case the tuples are ordered with respect to the left lexicographic ordering, so $(a_1,a_2,a_3,a_4) > (b_1,b_2,b_3,b_4)$ if $a_p > b_p$ and $a_q = b_q$ for all $q<p$. 
\end{defi}
Note that we are only keeping track of bi-gradings with non-zero ranks; likewise, the notation $(j,i,k,*)$ will mean that $*$ is some positive integer. For $\Theta(G,j)$ we group together all the $4$-tuples of the form $(j,i,k,*)$. See Example~\ref{ex:duegrafi} for an explicit computation of $\Theta(G,2)$. 

As we'll see in the rest of this section, the sequence $\Theta (G,j)$ for increasing $j$ can be thought of as providing approximated 
information on $G$, ranging from local to global. A similar yet unrelated theme will appear in Section~\ref{sec:uberhomology}.\\

\begin{rmk}
Even if a complete computation of all the possible horizontal homologies of the $\varepsilon$-colourings of $G$ becomes rapidly unwieldy --as there are $2^m$ colourings to be considered-- we will see in what follows that the type of information contained in small portions of these invariants is quite subtle, and is on average quite effective in distinguishing between non-isomorphic graphs (see Table~\ref{tab:data} and the discussion at the end of this section).

The computational cost of determining these horizontal groups with few black dots is polynomial in the number of vertices. 
There are also some tricks that can be used to speed up computations; a reasonably efficient --but far from optimal-- program computing $\Theta(G,j)$ and $\Delta$ is available at \cite{miogithub}. 
This program is not meant in any way to be a practical tool for the graph isomorphism problem, recently improved by Babai \cite{babai2016graph} (see the review paper \cite{grohe2020graph} and references therein). Rather, in the light of Conjecture~\ref{conj:dissimilarity}, it would be interesting to explore  the classification extent of $\Delta$, and the kind of information encoded in the horizontal homology for intermediate values of $|\varepsilon|$.
\end{rmk}

\begin{exa}\label{ex:duegrafi}
Here we can see an easy example of how the horizontal homology can be used to distinguish between two graphs; consider the two graphs shown in Figure~\ref{fig:twographs}; these have the same number of edges and are both $3$-regular, hence in particular have the same degree sequence. Therefore their horizontal homologies coincide for $|\varepsilon| = 0,1$ (and consequently $\Theta(G_1,j) = \Theta(G_2,j)$ for $j = 0,1$).

\begin{figure}[ht]
\centering
\includegraphics[width = 8cm]{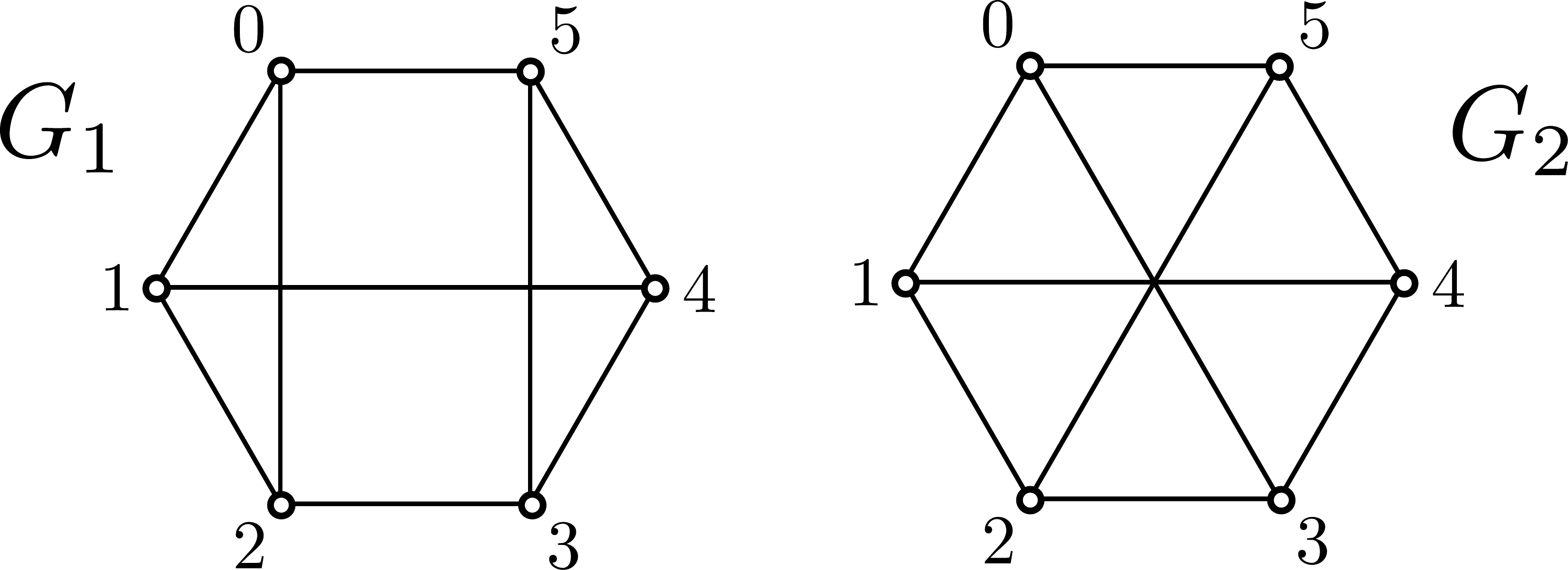}
\caption{Two graphs distinguished by their horizontal homologies with only $2$ black vertices.}
\label{fig:twographs}
\end{figure}

It is not hard to compute the groups $\HH^h(G_i, \varepsilon)$ for $i=0,1$ by hand or using the program in \cite{miogithub}
; it is in fact enough to restrict the computation to those colourings $\varepsilon \in \Z_2^6 $ with exactly two black vertices. We'll see in what follows that this implies that $\Delta(G_1,G_2) = \frac{2}{3}$. In the computations below, the notation $\HH^h(G, |\varepsilon|=2)$ denotes  the horizontal homology for all colours with two black vertices. The colour written on the right indicates a colour realising the corresponding group.

\begin{equation*}
\HH^h(G_1, |\varepsilon| = 2) = \begin{cases}
\F_{(1,2)}^4 \oplus \F_{(1,1)} \oplus \F_{(0,1)} \oplus \F_{(0,0)} & \text{ in $6$ cases  (\emph{e.g.}~} \varepsilon = (0,0,0,1,0,1))\\
\F_{(1,2)}^3 \oplus \F_{(1,0)}^2 \oplus \F_{(0,0)}^2  & \text{ in $6$ cases (\emph{e.g.}~} \varepsilon = (0,0,1,0,1,0))\\
\F_{(1,2)}^4 \oplus \F_{(0,0)}  & \text{ in $3$ cases (\emph{e.g.}~} \varepsilon = (0,0,1,1,0,0))
\end{cases}
\end{equation*}
\begin{equation*}
\HH^h(G_2, |\varepsilon| = 2) = \begin{cases}
\F_{(1,2)}^4 \oplus \F_{(0,0)} & \text{ in $9$ cases (\emph{e.g.}~} \varepsilon = (0,0,0,0,1,1))\\
\F_{(1,2)}^3 \oplus \F_{(1,1)}^3 \oplus \F_{(0,1)}\oplus \F_{(0,0)}^2  & \text{ in $6$ cases (\emph{e.g.}~} \varepsilon = (0,0,0,1,0,1))
\end{cases}
\end{equation*}

In particular 
\begin{equation*}
\Theta(G_1,2) = \left( (2,1,2,54),(2,1,1,6),(2,1,0,12),(2,0,1,6),(2,0,0,4) \right),    
\end{equation*}
\begin{equation*}
\Theta(G_2,2) = \left( (2,1,2,54),(2,1,1,18),(2,0,1,6),(2,0,0,21) \right).    
\end{equation*}
\end{exa}

\begin{rmk}
The invariant $\widehat{\Theta}(G)$ clearly contains a greater amount of information on $G$ than $\Theta(G)$. Indeed from the values of the invariant  $\widehat{\Theta}(G)$ we can uniquely reconstruct $G$, and in fact it is sufficient to only consider $\widehat{\Theta}(G,2)$; this is the collection of bi-degrees and ranks of the horizontal homologies of $G$ indexed by all $\varepsilon$ containing exactly two black vertices.

To see why, note that it is easy to recover $|V(G)|$ by looking at the length on any colour $\varepsilon$. We can then construct a graph $G^\prime$ whose vertices are the elementary colours $\varepsilon_p$ for $p = 1, \ldots, |V(G)|$. There is an edge in  $G^\prime$ between the vertices $\varepsilon_q$ and $\varepsilon_t$ whenever $\mathrm{rk} (\HH^h_0 (G, \varepsilon_{qt} ,0)) = 1$ (see the left part of Figure~\ref{fig:complexesgraphs}) where $\varepsilon_{qt} = \varepsilon_q + \varepsilon_t$; equivalently, we have an edge for any entry of the form $(\varepsilon_{qt}, 0,0,1) \in \widehat{\Theta}(G,2)$). Again, this follows from the fact that the horizontal homology of such a coloured graph, in bi-degree $(0,0)$ gives the number of connected components of the black subgraph, which is $1$ if and only if the two vertices are adjacent.
\end{rmk}
So the information contained in $\widehat{\Theta}(G,2)$ is more than enough to reconstruct $G$. This is not necessarily good news, since it implies that it cannot be used efficiently to distinguish graphs; in other words, relabelling the vertices of a graph would yield two different $\widehat{\Theta}$ invariants.

The same is not true for $\Theta(G)$. Indeed, in this case, since all the tuples are ordered (and we do not have the direct dependency of each tuple from the colouring), we are creating a ``canonical'' invariant, \emph{i.e.} not dependent on the labels of the vertices.

Several results and computations suggest the following restatement of Conjecture~\ref{conj:dissimilarity}.
\begin{conj}\label{conj:completegraph}
The collection $\Theta(G,j)$ for $j = 1, \ldots, |V(X)|$ is a complete graph invariant.
\end{conj}
We are now ready to give a proper definition of $\Delta$.
\begin{defi}
For a pair of simple and connected graphs $G_1$, $G_2$ with the same number of vertices $m$, define their \emph{dissimilarity} as the rational number in the unit interval $$\Delta (G_1,G_2) = 1- \frac{1}{m}\cdot\min\{j\ge 0\,|\, \Theta(G_1,j) \neq \Theta(G_2,j) \}$$
If instead the two graphs have a different number of vertices define  their dissimilarity as $\Delta (G_1,G_2) = \infty$.
\end{defi}
Heuristically, a small dissimilarity means that the two graphs share many features, while a value close to $1$ indicates that the two graphs are indeed rather different from a topological and\slash or combinatorial point of view.

\begin{prop}\label{prop:triangleineq}
If $G_1,G_2$ and $G_3$ are simple and connected graphs with the same number $m$ of vertices, then 
$$\Delta(G_1,G_2) \le \Delta(G_1,G_3) + \Delta(G_3,G_2).$$
\end{prop}
\begin{proof}
Let us start by assuming that the $\Theta$ invariants of $G_1$ and $G_2$ coincide up to colourings with exactly $r$ black vertices; this implies that $\Delta(G_1, G_2) = 1 - \frac{r}{m}$. Then we can distinguish two cases; in the first one, we assume that $\Theta(G_3,j) = \Theta(G_1,j)$ for all $0\le j\le s \le r$. Then necessarily $\Delta(G_1,G_2) \le \Delta(G_1,G_3)$. 

If instead we assume that $\Theta(G_3,j) = \Theta(G_1,j)$ for all $0\le j \le s^\prime \le m$ for some $s^\prime > r$. Then we get that $\Delta (G_2,G_3) = \Delta (G_1,G_2)$, and we are done.
\end{proof}

We say that $\Delta$ \emph{detects} a graph $G$ if $\Delta(G, G^\prime) >0$ for any  other graph $G^\prime$ not isomorphic to $G$. Since $\Theta(G,0)$ detects the numbers of edges, it follows from the fact that the complete graph $K_m$ is the unique simple graph with $m$ vertices and $\frac{m(m-1)}{2}$ edges that $\Delta$ detects $K_m$. More precisely we see that complete graphs are maximally distant from all other graphs sharing the same number of vertices: $$\Delta(K_m,G) = 1 \, \,\,\forall \,G \neq K_m.$$ A similar statement holds for cycle graphs $$\Delta(L_m,G) \ge \frac{m -1}{m} \, \,\,\forall \,G \neq L_m.$$ In this latter case, the value of $\Theta(L_m,0)$ alone might not be enough to detect $L_m$, so we need to add the information about the valence of each vertex --which is  contained in $\Theta(L_m,1)$. 
In a similar direction, it is easy to see that the class of regular graphs is distinguished from all other graphs; this follows from the fact that $G$ is regular if and only if $\Theta(G,1)$ consists of $|V(G)|$ copies of the same tuple. In particular, regular graphs have dissimilarity at least $\frac{m-1}{m}$ from all non-regular graphs (but of course two regular graphs can have lower values of $\Delta$).
~\\

We also have the following simple result, further showing how subtle information about $G$ is encoded in the collection of its horizontal homologies, and how it can be used to provide lower bounds on $\Delta$. In what follows, $girth(G)$ is the minimal length of a cycle in $G$ and $VC(G)$ is the minimal size of a vertex cover for $G$, \emph{i.e.}~a subset of $V(G)$ including at least one endpoint of every edge of the graph.  We will denote with $deg(G)$ the sequence (in decreasing order, and allowing for repetitions) of the degrees of the vertices in $G$.
\begin{lem}\label{lem:boundsondelta}
Let $G_1$ and $G_2$ be two simple and connected graphs with $m$ vertices.
\begin{itemize}
\item If $deg(G_1) \neq deg(G_2)$ then $\Delta(G_1,G_2) \ge \frac{m-1}{m}$.
\item If  $m- \alpha = girth (G_1) < girth(G_2)$, then $\Delta(G_1,G_2) \ge \frac{\alpha}{m}$.
\item If  $m- \beta = VC (G_1) < VC(G_2)$, then $\Delta(G_1,G_2) \ge \frac{\beta}{m}$.
\end{itemize}
\end{lem}
\begin{proof}
Since we can read the array $deg(G)$ by computing $\mathrm{rk}( \HH^h(G, (0, \ldots, 0), 2)) - \mathrm{rk} (\HH^h(G, \varepsilon_p,2))$ for each $p = 1, \ldots, m$, the first item follows. Note that we do not need to have the information of which vertices correspond to which values of $p$ to conclude.

For the second one, note that if $girth(G_1) = m-\alpha$ is less than $girth(G_2)$, we can distinguish the two graphs by looking at $\Theta(G,m-\alpha)$: for $G= G_1$ we get at least one four-tuple of the form $(m-\alpha, 1,0,*) \in \Theta(G_1,m-\alpha)$, while there are none in $\Theta(G_2,m-\alpha)$. 

Lastly, it is easy to argue that the vertex cover number of a graph can be equivalently defined as $$VC(G) = \min \{j\ge 0\,|\, (j, 1, 2,*) \notin \Theta(G,j)\}.$$ 
This is because if $\varepsilon$ is the colouring where all vertices of a given (minimal) vertex cover are black, by definition there are no edges with both white endpoints.
Now, similarly to the previous point,  the fact that $VC$ is detected by  $\Theta$ implies the result.
\end{proof}
The last item in Lemma~\ref{lem:boundsondelta} can be generalised, yielding a bijection 
\begin{equation*}
\{\varepsilon \in \Z_2^m \,|\, \HH^h(G, \varepsilon,2) = 0\} \longleftrightarrow \{\text{vertex covers of } G\}.
\end{equation*}
Furthermore this is a bijection of posets. Recall that finding vertex covers of a given graph is a classical example of a NP-complete problem~\cite{karp1972reducibility}.\\

There is a further type of structure which is detected by $\Theta(G)$:
\begin{defi}
Let $G$ be a simple and connected graph. A tree $T \subseteq G$ is said to be \emph{spacious} if  $T$ has at most one edge in each cycle of length $3$ in $G$.
Moreover $T$ is \emph{maximal} if it is not a sub-tree of any other spacious tree in $G$ (it might however be the sub-tree of a non-spacious tree in $G$, see Figure~\ref{fig:largetrees}).
\begin{figure}[ht]
\centering
\includegraphics[width = 8cm]{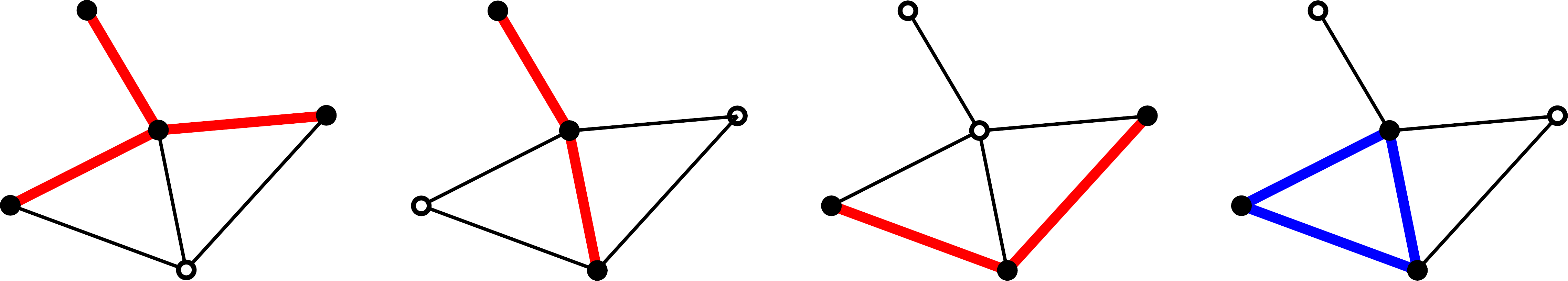}
\caption{In red, the only three maximal spacious trees for a graph with five vertices, together with their associated colourings. Note that maximal spacious trees do not necessarily have the same number of edges. On the right, we see that if two edges in a length $3$ cycle are black, then the induced black subgraph cannot be a tree.}
\label{fig:largetrees}
\end{figure}
\end{defi}
Equivalently, a tree is spacious if colouring all vertices of $T$ in black produces a black subgraph $Bl(G,\varepsilon)$ which is still a tree.

\begin{lem}
There exists a bijection 
\begin{equation*}
\left\{ \text{spacious trees in } G\right\} \longleftrightarrow \left\{ \varepsilon \in \Z_2^m\,|\, \HH^h_0(G, \varepsilon,0) \cong \F \text{ and }\HH^h_1(G, \varepsilon,0) \cong 0\right\}
\end{equation*}
\end{lem}
\begin{proof}
Colour all vertices of a spacious tree $T$ in black, obtaining a colouring $\varepsilon_T$. Then the black graph $Bl(G, \varepsilon_T)$ is a tree, thus $\HH^h_0(G, \varepsilon_T,0) \cong \HH_0(Bl(G,\varepsilon_T)) \cong \F$ and  $\HH^h_1(G, \varepsilon_T,0) \cong \HH_1(Bl(G,\varepsilon_T)) \cong 0$. 
The converse is immediate, once we notice that the conditions on the horizontal homologies in filtration degree $0$ imply that $Bl(G, \varepsilon)$ has to be a tree. This tree is necessarily spacious (see the right part of Figure~\ref{fig:largetrees}).
\end{proof}
As a consequence, graphs with a different number of spacious trees (or with spacious trees with a different number of edges) must have positive dissimilarity. \\

It is natural to ask the efficiency and computational complexity of computing the dissimilarity; using a Sage program (available at~\cite{miogithub}), we were able to show that the dissimilarity is quite big for all graphs with less than $10$ vertices, and for trees with up to $16$ vertices. This implies that the $\Theta$ invariants (for small values of $j$) are indeed sufficient to identify all such graphs. 

\begin{table}[ht]
\begin{center}
\small
\begin{tabular}{|c|c|c|c|c|c|c|c|c|}
\hline
$|V|$ & $<10$ & $10$ & $11$ & $12$ & $13$ & $14$ & $15$ &$16$\\
\hline
\# trees & 95 & 106 & 235 & 551 & 1301 & 3159 & 7741 & 19320  \\
\hline
\hline
$|\varepsilon| = 3$ & -- & $ \perc{1.77}{2}$ & $\perc{1.62}{2}$ & $\perc{1.58}{2}$ &  $\perc{1.26}{2}$ &  $\perc{8.76}{3}$ &  $\perc{5.78}{3}$ &  $\perc{3.65}{3} $\\
\hline
$|\varepsilon| = 4$ & -- & -- & -- & $\perc{3.29}{4}$ & $ \perc{5.9}{5}$ &  $\perc{6.01}{5}$ & $ \perc{2.5}{5}$ &    $\perc{3.21}{6}$\\
\hline
\end{tabular}
\vspace{0.3cm}
\caption{Discriminatory power of $\Delta$ for all trees with up to $16$ vertices. The first row is the number of vertices, while the second is the number of distinct trees with the given number of vertices.
Values in the table denote the probability (rounded up) that a randomly chosen pair of trees is not detected by $\Theta$ in level $|\varepsilon|$. 
}
\label{tab:datatrees}
\end{center}
\end{table}

\begin{table}[ht]
\begin{center}
\begin{tabular}{|c|c|c|c|c|c|c|}
\hline
$|V|$   & $7$ & $8$ & $9$\\
\hline
\# graphs  & $853$ & $1117$ & $261080$\\
\hline
\hline
$|\varepsilon| = 3$  & -- & $\perc{4.0}{4}$ & $ \perc{4.43}{5}$\\
\hline
$|\varepsilon| = 4$  & -- & -- &  $ \perc{2.97}{7}$\\
\hline
$|\varepsilon| = 5$  & -- & -- & $ \perc{1.46}{9}$\\
\hline
$|\varepsilon| = 6$  & -- & -- & $\perc{2.93}{9}$\\
\hline
\end{tabular}
\vspace{0.3cm}
\caption{Results of the computation of $\Delta$ for all graphs with less than $10$ vertices.}
\label{tab:data}
\end{center}
\end{table}
It is apparent that --at least in the range we considered above-- the probability of finding two different graphs with small dissimilarity is extremely low, and drastically decreases for increasing values of $|\varepsilon|$.

In theory the computational cost of computing $\Delta$ is generally rather high as a function of the number of vertices $m$; first for $\Theta(G,j)$ there are $\binom{m}{j}$ possible colours with weight $j$, and we need to compute the horizontal homology of $G$ for each of them. Furthermore, we have to perform the same computations for each graph considered. 
However, it appears that for most practical purposes it is sufficient to restrict our computations to $j\le 3$. Moreover, since the information contained in $\Theta(G,j)$ for $j = 0,1$ can be recovered easily without the need to compute any horizontal homology group, we can just restrict to $j=2,3$. Therefore we only need to compute $\frac{m(m-1)}{2}$ homology groups for each graph whenever $\Delta = \frac{m-2}{m}$, and $\frac{1}{6}(m^3 - 3m +2)$ if $\Delta = \frac{m-3}{m}$. Of course, it is possible to find examples of pairs of graphs with higher dissimilarity, as shown in Figure~\ref{fig:graphsbigdiss} and Table~\ref{tab:data}. However the percentage of such pairs (when randomly sampled in the space of graphs with a given number of vertices) appears to be extremely low. As an example, there are exactly two different pairs of graphs with $9$ vertices that have dissimilarity $\frac{1}{3}$ among the possible $34,081,252,660$ pairs of distinct graphs with up to $9$ vertices (and these achieve the minimal dissimilarity among all such pairs, see Table~\ref{tab:data}).

\begin{figure}[ht]
\centering
\includegraphics[width=12cm]{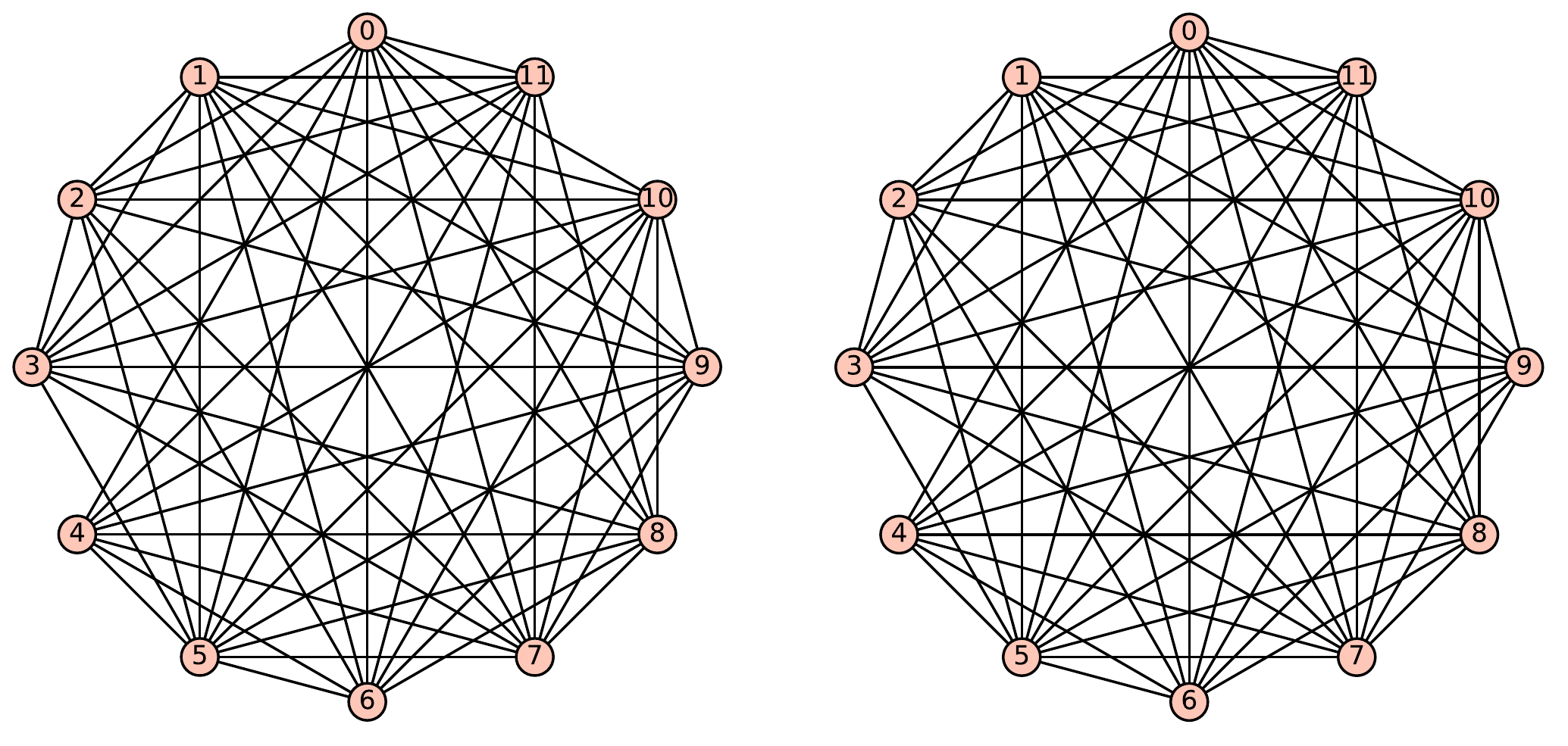}
\caption{Two graphs on $12$ vertices with dissimilarity $\frac{2}{3}$.}
\label{fig:graphsbigdiss}
\end{figure}
There is a sequence of extremely simple pairs of graphs with a relatively small dissimilarity; for any positive integer $n$, consider two loop graphs on $2n$ vertices. To obtain the first graph $G_1(n)$ connect one vertex to its antipode, and for $G_2(n)$ connect a vertex to one of the two neighbours of its antipode, as in Figure~\ref{fig:circleswithdiameter}.
\begin{figure}[ht]
\centering
\includegraphics[width=8cm]{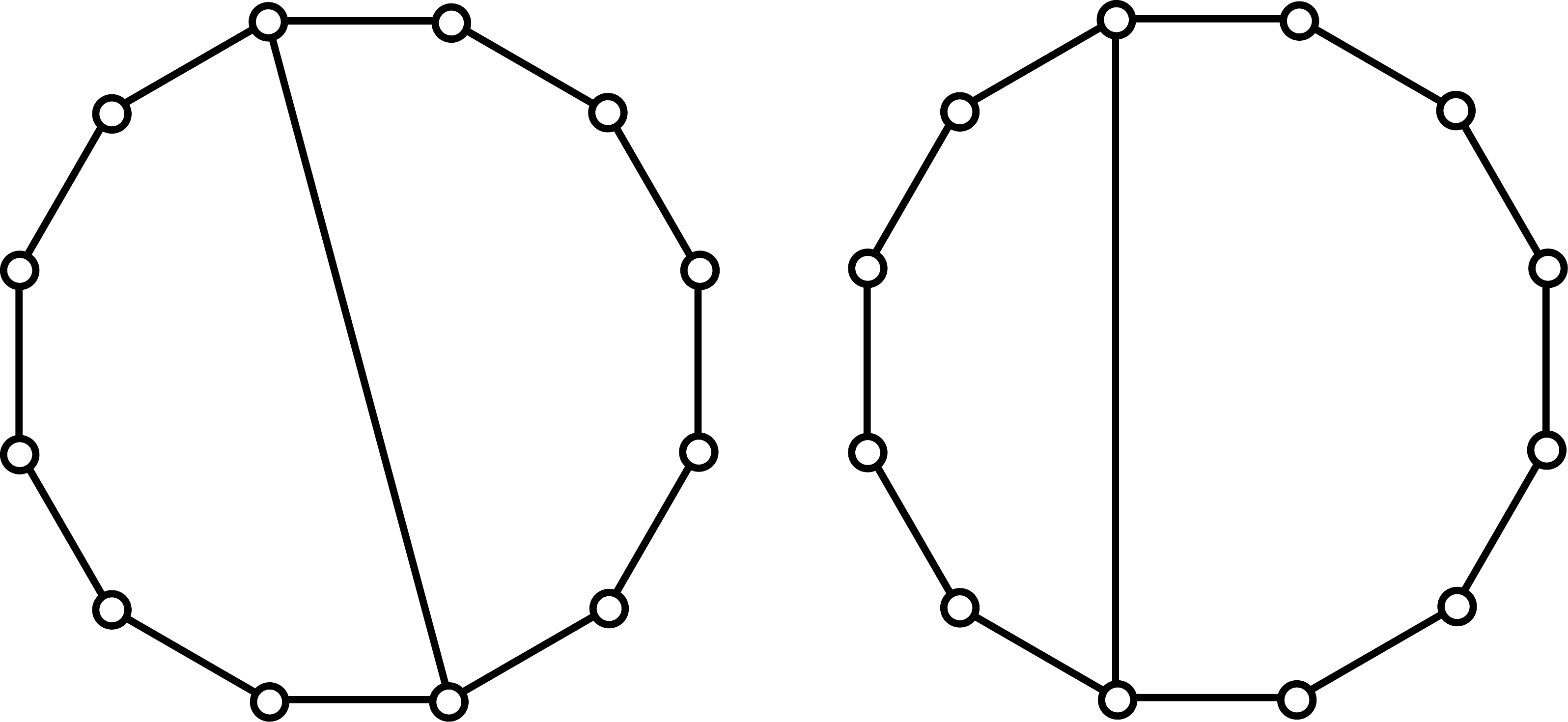}
\caption{The two graphs $G_1(6)$ and $G_2(6)$; each pair in these two families of graphs has dissimilarity close to $\frac{3}{4}$.}
\label{fig:circleswithdiameter}
\end{figure}
It can be proved that the dissimilarity between these pairs is $\Delta(G_1 (n),G_2 (n)) = 1 - \frac{1}{2n}\cdot \lceil \frac{n}{2}\rceil$.

\section{\"Uberhomology}\label{sec:uberhomology}

In this section we outline a procedure to organise the whole poset of $\varepsilon$ colourations of a finite and connected simplicial complex $X$ with $m$ vertices, and exploit the structure of the horizontal homology developed in the previous sections to define a further homology theory, independent of the colouring.

Firstly let us observe that if we are looking at all possible colourations, using  Remark~\ref{rmk:isohordiag} it makes sense to only consider  horizontal homologies, in order to avoid repetitions. 

Consider the $1$-skeleton of the $m$-cube with vertices $\varepsilon \in \Z_2^m$, and decorate each vertex with the corresponding horizontal homology group $\HH^h(X,\varepsilon)$. A schematic picture of a cube for a simplicial complex with three vertices is shown in Figure~\ref{fig:cubo}.

We are now going to  to combine the information contained in the decorated $m$-cube in a way that is reminiscent of Khovanov's cube of resolutions \cite{khovanov2000categorification} for link diagrams. See also \cite{bar2002khovanov} for a simplified approach to Khovanov's homology, and \cite{helme2005categorification}, \cite{loebl2008chromatic}, \cite{everitt2009homology} for related constructions.\\

Consider an edge of the cube; if it connects $\varepsilon$ to $\varepsilon^\prime$, denote it by $\eta$, where $\eta(i) = \varepsilon(i)$ whenever $\varepsilon(i) = \varepsilon^\prime (i)$, and it equals a symbol $*$ in the only position where they differ. The edges are oriented in the direction of increasing $|\varepsilon|$ (see Figure~\ref{fig:cubo}).

\begin{figure}[ht]
\begin{tikzcd}[row sep=1.5em, column sep = 1.5em]
&  & {\HH^h(X,(1,0,0))} \arrow[rr, "{d_{(1,*,0)}}" ] \arrow[rrdd, "{d_{(*,1,0)}}" ]   &  & {\HH^h(X,(1,1,0))} \arrow[rrdd, "{d_{(1,1,*)}}"] & &\\
& & & & & &\\
{\HH^h(X,(0,0,0))} \arrow[rruu, "{d_{(*,0,0)}}"] \arrow[rr, "{d_{(0,*,0)}}" ] \arrow[rrdd, "{d_{(0,0,*)}}" '] &  & {\HH^h(X,(0,1,0))} \arrow[rruu, "{d_{(1,0,*)}}"'] \arrow[rrdd, "{d_{(*,0,1)}}"] & & {\HH^h(X,(1,0,1))} \arrow[rr, "{d_{(1,*,1)}}"] & & {\HH^h(X,(1,1,1))} \\
& & & & & &\\
&  & {\HH^h(X,(0,0,1))} \arrow[rruu, "{d_{(0,1,*)}}"'] \arrow[rr, "{d_{(0,*,1)}}" '] &  & {\HH^h(X,(0,1,1))} \arrow[rruu, "{d_{(*,1,1)}}" '] & &
\end{tikzcd}
\caption{The poset structure for the filtered horizontal homologies of a simplicial complex with $3$ vertices.} 
\label{fig:cubo} 
\end{figure}
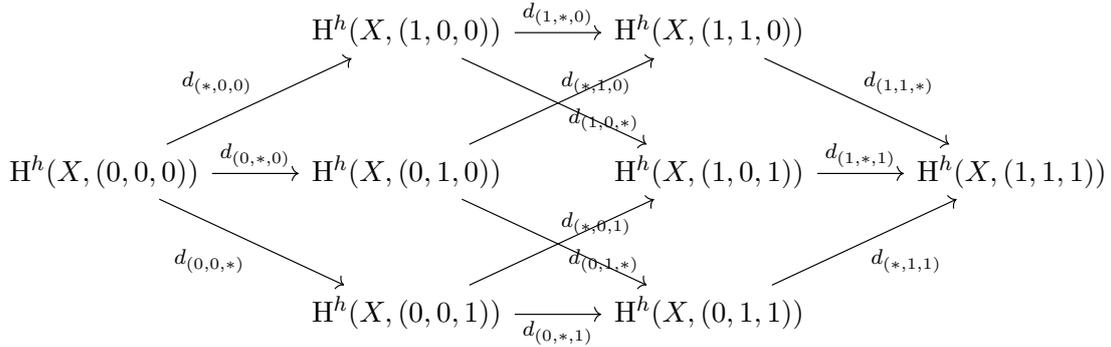

To each such cube edge $\eta$, we associate a map 
\begin{equation}\label{eqn:deta}
d_\eta: \HH_*^h(X,\varepsilon) \longrightarrow  \HH_*^h(X,\varepsilon^\prime). 
\end{equation}

The definition of these maps is quite straightforward; first consider the identity map $\mathrm{Id} : C(X) \longrightarrow C(X)$, and note that if we endow the complexes with the two filtrations $w_\varepsilon$ and $w_{\varepsilon^\prime}$ induced by the adjacent vertices $\varepsilon$ and $\varepsilon^\prime$, then $\mathrm{Id} : C(X, \varepsilon) \longrightarrow C(X, \varepsilon^\prime)$ is a filtered map. Namely, $w_\varepsilon (\sigma) \le w_{\varepsilon^\prime}$. The map $d_\eta$ associated to the cube's edges is the map induced on the horizontal homology by the weight-preserving part of these filtered identity maps.

In other words, to define these maps, write $[x]_\varepsilon$ for a generator of $\HH^h(X, \varepsilon)$, with the understanding that $[x]_\varepsilon$ is represented by $x \in C(X, \varepsilon)$, a linear combination of simplices with the same bi-grading. Then the value of $d_\eta ([x]_\varepsilon)$ is given by $[x]_{\varepsilon^\prime}$, the class represented by the same exact sum of simplices $x$ in the homology group $\HH^h(X, \varepsilon^\prime)$.

It turns out that, in order to get well-definedness of these maps (proved in Proposition~\ref{prop:dsquare} below), we need to impose that the filtration degree is preserved, namely if $w_\varepsilon (x) \neq w_{\varepsilon^\prime} (x)$, then $d_\eta ([x]_\varepsilon)=0$, which implies that $d_\eta (\HH^h_i(X, \varepsilon, k)) \subseteq \HH^h_i(X, \varepsilon^\prime, k) $. Note that, in the discussion above, it is not restrictive to only consider homogeneous generators, as with this definition $d$ is linear and preserves the bi-degree of the generators. 

It is possible to give an alternative definition of the maps $d_\eta$; the filtered identity map $\mathrm{Id} : C(X, \varepsilon) \longrightarrow C(X, \varepsilon^\prime)$ induces a map of spectral sequences. Then $d_\eta$ is simply the induced map between the associated graded objects.

We can then define for $0\le j\le m$ $$\ddot{C}^j(X) = \bigoplus_{|\varepsilon| = j} \HH_*^h(X,\varepsilon),$$
and $$d^j = \bigoplus_{|\eta| = j} d_\eta.$$
Here $|\eta|$ is the sum of the entries in $\eta$ that are not decorated with $*$.\\

Before proving that $d$ is in fact a well-defined differential on $\ddot{C}(X)$, let us give a brief description of the effect that it has on a simple example; when passing from $\HH^h(X, \varepsilon)$ to $\HH^h(X, \varepsilon^\prime)$ (where again $\varepsilon$ and $\varepsilon^\prime$ are two adjacent vertices in the cube) if a simplex $\sigma \subseteq X$ contains the unique vertex that becomes black, we are adding one component to its horizontal boundary. 
Namely, the horizontal boundary of $\sigma$ in $C(X, \varepsilon^\prime)$ will differ from $\partial_h (\sigma)$ in $C(X, \varepsilon)$ by the addition of the face of $\sigma$ obtained by removing the new black vertex. \\
In this case the value on $[\sigma]_\varepsilon$ of the map $d_\eta$ associated to the edge from $\varepsilon$ to $\varepsilon^\prime$ is $0$, as the filtration degree is not preserved. 

One might wonder if a cycle can be sent to a non-cycle through the differentials $d$; let $x \in C(X, \varepsilon)$ be a cycle for the horizontal differential, so\footnote{Here we have highlighted the dependency of the horizontal differential on the colouring.} 
$\partial_h^\varepsilon (x) = 0$. Assume that $\varepsilon^\prime$ is obtained from $\varepsilon$ by colouring in black the vertex $v$ via the edge $\eta$. Then it turns out that requiring that $x$ is a cycle in $C(X, \varepsilon^\prime)$, is equivalent to requiring that $d_\eta$ preserves the $w$ degree. This is one of the main technical facts that allows for the good definition of the \"uberdifferential. 

\begin{lem}\label{lem:techinical}
Let $\eta$ be the edge between $\varepsilon$ and $\varepsilon^\prime$ induced by colouring in black the vertex $v$, and $x = \sum_{r =1}^p \sigma_r$ be a homogeneous cycle in $\left( C(X,\varepsilon), \partial_h \right)$. Then $x$ is a cycle in $\left( C(X, \varepsilon^\prime), \partial_h \right)$ if and only if $v$ does not belong to $\cup_{r=1}^p V(\sigma_r)$.
\end{lem}
\begin{proof}
If the vertex $v$ becoming black along $\eta$ is not in $\cup_{r=1}^p V(\sigma_r)$, then clearly $\partial_h^\varepsilon (x) = \partial_h^{\varepsilon^\prime} (x) = 0$, since the horizontal boundary only depends on the coloured simplices representing $x$.

Let us consider what happens instead if $v \in \cup_{r=1}^p V(\sigma_r)$; firstly $x$ might not be homogeneous in $\left( C(X, \varepsilon^\prime), \partial_h \right)$. Let us write $x = \sum_{r = 1}^{p_0} \sigma_r^0 +\sum_{r = 1}^{p_1} \sigma_r^1 $, where $w_{\varepsilon^\prime} (\sigma_r^0) = w_{\varepsilon} (x) $ and $w_{\varepsilon^\prime} (\sigma_r^1) = w_{\varepsilon} (x) -1$. Let us focus on the latter simplices; their horizontal boundary, gains a new codimension one face for each simplex $\sigma_r^1$, when the colouring changes from $\varepsilon$ to $\varepsilon^\prime$. 
We can explicitly identify these new faces appearing with the simplices that are opposite to $v$ in each $\sigma_r^1$. Therefore $\partial_h^{\varepsilon^\prime} (x) \neq 0$, and $x$ is not a cycle in $\left( C(X, \varepsilon^\prime), \partial_h \right)$.
\end{proof}

\begin{prop}\label{prop:dsquare}
The chain map $d$ is a bi-grading preserving differential.
\end{prop}
\begin{proof}
Let us begin by noting that $d$ is a homomorphism. This follows from the definition of the maps $d_\eta$, as these are obtained as the grading-preserving component of a filtered homomorphism.

Next we can show that $d$ is well-defined, \emph{i.e.} that its values do not depend on the chosen representative $[x]_\varepsilon \in \HH^h(X, \varepsilon)$. As $d$ preserves the $(dim, w)$ bi-grading (since all maps $d_\eta$ do by definition), we can assume that $[x]_\varepsilon$ is represented by a sum of simplices with the same bi-grading. Let $y = \sum_{r=1}^p \tau_r$ be a cycle represented by a sum of simplices in the same bi-degree as $x$, representing the trivial class in $\HH^h(X, \varepsilon)$; in other words there exists a $z = \sum_{r = 1}^q \sigma_r$ such that $y = \partial^\varepsilon_h \left(z\right)$ in $\left( C(X, \varepsilon), \partial^\varepsilon_h \right)$, for simplices $\sigma_r$ with $w_\varepsilon(\sigma_r) = w_\varepsilon (y)$ and $dim(\sigma_r) = dim(y)+1$. 

In order to prove that $d([x + y]_\varepsilon) = d([x]_\varepsilon)$, it suffices to show that this holds for each $d_{\eta}$ whose domain is $\HH^h(X, \varepsilon)$, as these are the only components of the differential whose domain contains these elements. Assume that $\eta$ is the edge connecting $\varepsilon$ to $\varepsilon^\prime$, and $v$ is the vertex that becomes black along $\eta$. 
If $v$ does not belong to the vertices of the simplices in $y$ and $z$, then both $\partial^{\varepsilon^\prime}_h(y) = 0$ and $\partial^{\varepsilon^\prime}_h(z) = y$. Thus, $[x + y]_{\varepsilon^\prime} = [x]_{\varepsilon^\prime}$, and the statement follows in this case.

If instead the black vertex $v$ appears in some $\tau_r$, there are two possibilities to consider; if $v$ belongs to all the simplices $\tau_r$ that compose $y$, then the weight of $y$ decreases by one. Therefore in this case, $d_\eta ([y]_\varepsilon) = 0$ by definition. If however $v$ does not belong to all simplices in $y$, then $y$ is not a homogeneous generator in $C(X, \varepsilon^\prime)$. We can nonetheless prove that the simplices composing $y$ that do not contain $v$ still form a boundary in $C(X, \varepsilon^\prime)$ (this is the part of $y$ that ``survives'' in the passage from the colouring $\varepsilon$ to $\varepsilon^\prime$). To this end, write $y^\prime = \sum_{s \text{ such that } v\notin \tau_s} \tau_s$ and $z^\prime = \sum_{r \text{ such that } v\notin \sigma_r} \sigma_r$. Then, by definition, $\partial^{\varepsilon^\prime}_h (z^\prime) = y^\prime$.



We are now left to prove that $d$ squares to $0$. This will come as a consequence of the commutativity of the components of the differential along square faces of the cube. More precisely, the definition of the components of the differential $d$ implies that the two orders for colouring in black two vertices give contributions that cancel each other out.

Note that for any $[x]_\varepsilon  \in \HH^h(X, \varepsilon)$, the composition  $d^{|\varepsilon|+1} \circ d^{|\varepsilon|} ([x]_\varepsilon)$ is a sum of elements of the form $d_{\eta} \circ d_{\eta^\prime}([x]_\varepsilon)$, where $\eta$ and $\eta^\prime$ correspond to two consecutive edges of the cube, or equivalently to two different $0$ entries in $\varepsilon$. Assume that (using the notation of Figure~\ref{fig:cubo}) $\eta$ and $\eta^\prime$ have a $*$ respectively in the indices $a$ and $b$ in the set $\{r \:|\;\varepsilon(r) = 0\}$. 

Then their contribution is cancelled over $\F$ by the composition of the maps corresponding to travelling first along the edge $\widetilde{\eta}$ and then $\widetilde{\eta}^\prime$, with $*$ in the positions $b$ and $a$ respectively.
\end{proof}

\begin{defi}\label{def:uberh}
The \emph{\"{u}berhomology} of a finite and connected simplicial complex $X$, denoted by $\uH_*(X)$, is the homology of the complex $\left( \ddot{C}^*(X), d^*\right)$. Since the differential preserves the $(dim, w)$ bi-degree, $\uH(X)$ is triply graded:
$$\uH(X) = \bigoplus_{\substack{j = 0,\ldots,\,|V(X)|\\i = 0,\ldots,\,dim(X) \\ k = 0, \ldots,\,dim(X)+1}} \uH_i^j(X,k).$$
We will refer to $j,i,k$ as the homological, dimension and filtration degrees respectively.
\end{defi}

\begin{rmk}
The proof of Proposition~\ref{prop:dsquare} uses in a crucial way the choice of $\F$ coefficients. Nonetheless, adopting the categorical approach of~\cite{ubergrafi} it is possible to show that the same framework can be extended to general coefficients.
\end{rmk}

\begin{rmk}
Up to isomorphisms, the \"uberhomology does not depend on the ordering of the vertices of $X$ or on the colourings, hence it is an invariant of the simplicial complex. Unlike ``regular'' simplicial homology, it is not a homotopy invariant of $|X|$, the geometric realisation of $X$, as we will see in Section~\ref{sec:computations}. We will however prove in Theorem~\ref{thm:topdim} that, if $X$ is a triangulated manifold, its top-dimensional \"uberhomology is in fact independent of the triangulation.
\end{rmk}

\begin{rmk}
As noted already in Lemma~\ref{lem:techinical} the requirement that each map $d_\eta$ is $0$ whenever the $w$ degree of a generator decreases cannot be relaxed. A simple example is provided in Figure~\ref{fig:ubercubefor1simplex}: without this requirement, the value of $d^2([v_1]_{(0,0)}) = d_{(1,*)}\circ d_{(*,0)}([v_1]_{(0,0)}) + d_{(*,1)}\circ d_{(0,*)}([v_1]_{(0,0)})$ would coincide with $[v_1]_{(1,1)}$ (since $d_{(0,*)}([v_1]_{(0,0)}) = 0$) which is the unique generator of $\HH^h(\Delta^1, (1,1))$, and in particular $d$ would not square to $0$.
\end{rmk}

\begin{exa}\label{ex:1spx}
\begin{figure}[ht]
\includegraphics[width = 13cm]{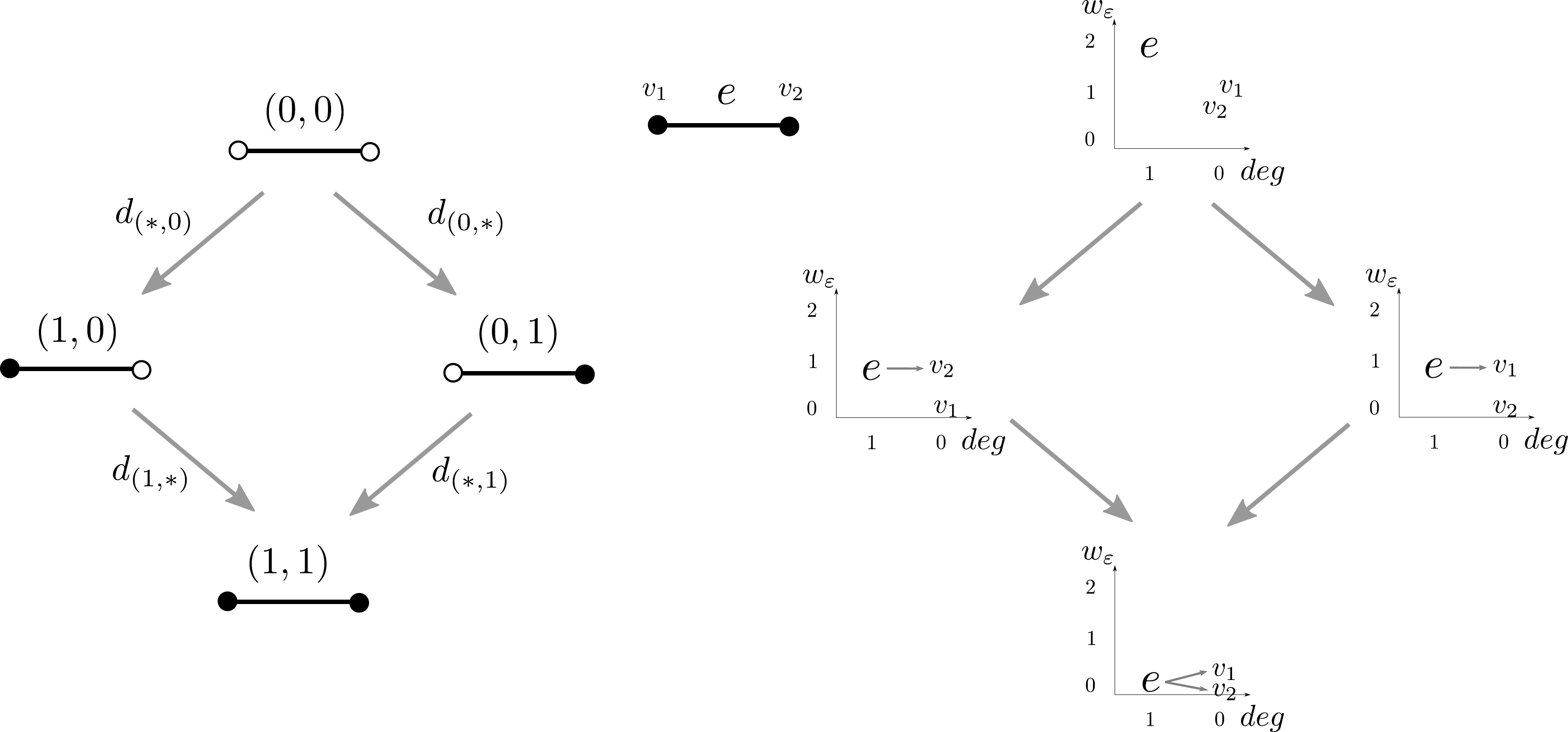}
\caption{On the left, the poset structure for the colourings of the $1$-simplex $\Delta^1$; on the right the associated horizontal complexes. The direct sum over each ``row'' on the left is the complex $\ddot{C}^j(\Delta^1)$, where  $j = |\varepsilon|$.}
\label{fig:ubercubefor1simplex}
\end{figure}
The \"{u}bercomplex of $\Delta^1$ is as follows (see Figure \ref{fig:ubercubefor1simplex}): 
\begin{align*}
\ddot{C}^0(\Delta^1) = &\F\langle [e]_{(0,0)}\rangle_{(1,2)} \oplus  \F\langle [v_1]_{(0,0)}\rangle_{(0,1)}  \oplus \F\langle [v_2]_{(0,0)}\rangle_{(0,1)} \\
\ddot{C}^1(\Delta^1) = &\F\langle [v_1]_{(1,0)}\rangle_{(0,0)}  \oplus \F\langle [v_2]_{(0,1)}\rangle_{(0,0)} \\  
\ddot{C}^2(\Delta^1) = &\F\langle [v_{1}]_{(1,1)}\rangle_{(0,0)}
\end{align*}
Here $\F^c\langle [t]_{\varepsilon}\rangle_{(a,b)}$ means that the horizontal homology has a rank $c$ component generated by the simplex $t$ in $\HH^h(X, \varepsilon)$, in bi-degree $(a,b)$.
Also, in $\ddot{C}^2(\Delta^1)$, note that $\F\langle [v_{1}]_{(1,1)}\rangle_{(0,0)} = \F\langle [v_{2}]_{(1,1)}\rangle_{(0,0)}$ as these generators are homologous in $\HH^h(\Delta^1, (1,1))$.

The only non-trivial \"{u}berhomology groups are (\emph{cf.}~Corollary~\ref{cor:uberofsimplices}): $$\uH^0(\Delta^1) \cong \F\langle [e]_{(0,0)}\rangle_{(1,2)} \oplus  \F\langle [v_1]_{(0,0)}\rangle_{(0,1)} \oplus  \F\langle [v_2]_{(0,0)}\rangle_{(0,1)},$$ and $$\uH^1(\Delta^1) \cong \F\langle [v_1]_{(1,0)} + [v_2]_{(0,1)}\rangle_{(0,0)}. $$ Note that in this latter case the generator is a sum of elements from horizontal homologies with different $\varepsilon$ colouring.
\end{exa}

\begin{rmk}\label{rmk:splitting_degrees}
Since the \"uberdifferential preserves both the filtration and homological gradings, we can think of $\uH(X)$ as a collection of groups $\uH_i^*(X, k)$ indexed by pairs $(i,k) \in \{0,\ldots, dim(X)\} \times \{0,\ldots, dim(X)+1\}$. This will become especially relevant in Section~\ref{sec:ubergraphs}.
\end{rmk}

\section{Properties and computations}\label{sec:computations}

We collect here some computations of the \"uberhomology for certain classes of simplicial complexes, and for infinite families in  specific cohomological degrees. Furthermore, we prove some general properties of $\uH(X)$, and show that it contains both topological and combinatorial information about $X$.\\

In what follows we will sometimes refer to certain topological spaces known as homology manifolds; a \emph{$n$-homology manifold} $X$ is a simplicial complex of dimension $n$ such that the link of every simplex $\sigma \subseteq X$ has the homology of a $(n-dim(\sigma))$-dimensional sphere. The class of homology manifolds contains all triangulated manifolds, namely those topological manifolds that are homeomorphic to the geometric realisation of a simplicial complex.\\

As is the case with the ordinary Khovanov homology for knots, it is possible to give in some cases an intrinsic description of the \"uberhomology groups in the $0$-th and top degrees; this is the content of Theorem \ref{thm:0dim} and Theorem \ref{thm:topdim}.
\begin{thm}\label{thm:0dim}
Given a connected simplicial complex $X$ with vertices $v_1, \ldots, v_m$, we have $$\uH^0_j (X) \cong \F^{n_j}_{(j, j+1)},$$
where $n_j$ is the number of $i$-dimensional simplices in $$\bigcap_{i = 1}^m \;\overline{St(v_i)}.$$
\end{thm}
\begin{proof}
It follows from Definition \ref{def:uberh} that   
\begin{equation}\label{eqn:0uberhom}
\uH^0(X) = \ker(d^0)  = \bigcap_{i = 1}^m \ker(d_{\eta_i}).
\end{equation} 
Here $d_{\eta_i}$ is the map corresponding to the edges $\eta$ whose only non-zero component is in position $i$.
It then suffices to prove that $\ker(d_{\eta_i})$ is generated by the simplices in $\overline{St(v_i)}$. As we will show below, this follows from Theorem~\ref{thm:dmmandsubgraphs}, after noting that all the colourings considered are elementary and thus dalmatian. 

More precisely, the only horizontal differentials in $C(X, \varepsilon_i)$ are those arising from the Morse matching induced by the elementary colouring $\varepsilon_i$ (see Figure~\ref{fig:outwardflow}). So all generators in each $C(X, \varepsilon_i)$ are either uniquely paired with another generator (with which they form an acyclic complex), or are not paired with any other generator. These latter elements are non-trivial in $\HH^h(X, \varepsilon_i)$, and, with the exception of the black vertex $v_i$, their bi-degree is the same as their counterparts in $\HH^h(X, (0,\ldots,0))$ (\emph{cf}. Equation~\eqref{eqn:zeroepsilon}), as they do not contain the black vertex. 

It is also immediate to note that $[v_i]_{(0,\ldots,0)}$ does belong to $\ker(d_{\eta_i})$, as its $w$ degree decreases by one.

So only the matched elements together with the black vertex  belong to $\ker(d_{\eta_i})$, and it is immediate to see that the associated simplices are the only ones containing a black vertex, or in other words $\ker(d_{\eta_i}) = \overline{St(v_i)}$.  The result now follows from Equation~\eqref{eqn:0uberhom}.
\end{proof}

The previous result implies at once the non-triviality of the $0$-degree part of the \"uberhomology for specific classes of simplicial complexes; denote by $X^\prime$ the barycentric subdivision of $X$.
\begin{cor}\label{cor:subdivisiondegree0}
We have that $\uH^0(X^\prime) = 0$, unless $X = \Delta^n$ for some $n\ge 1$. In this case, $\uH^0((\Delta^n)^\prime) = \F_{(0,1)}$, generated by the barycentre in $Int(\Delta^n)$.
\end{cor}
\begin{proof}
If $X = \Delta^n$, then the intersection of the closed stars of the vertices in $(\Delta^n)^\prime$ consists of a single point, the barycentre of $(\Delta^n)^\prime$. In particular, this implies that the intersections of the closed stars in a subdivision are composed by single points in the interior of each simplex in $X$; so if $X$ contains more than one simplex, by Theorem~\ref{thm:0dim} the $0$-degree \"uberhomology of $X$ vanishes. 
\end{proof}
\begin{rmk}\label{rmk:othersubdivisions}
The statement of Corollary~\ref{cor:subdivisiondegree0} can be extended to encompass other kinds of subdivisions (\emph{e.g.}~stellar or chromatic subdivisions --see \cite{kozlov2012chromatic}) with minor modifications, but for the sake of clarity we restricted to barycentric subdivisions. In the opposite direction, it is not hard to realise that the two minimal triangulations for $S^1\times S^1$ and $\R\mathbb{P}^2$ from Figure~\ref{fig:dalmatching} have non-trivial $\uH^0$; in both cases this is generated by the set of all vertices in the triangulation.
\end{rmk}
The \emph{diameter} $diam(X)$ of $X$ is the maximal geodesic distance among all pairs of vertices in the $1$-skeleton of $X$, endowed with the path metric.

\begin{cor}\label{cor:diameter}
If $X$ is a connected and finite simplicial complex, then  $$ diam(X) \ge 3 \Longrightarrow \uH^0(X) = 0. $$
\end{cor}
\begin{proof}
This is a direct consequence of Theorem~\ref{thm:0dim}, after noting that $$diam(X) \ge 3 \Longrightarrow \bigcap_{i=1}^m \overline{St(v_i)} = \emptyset,$$ since the closures of the stars of two vertices at distance at least three have trivial intersection.
\end{proof}

It is nonetheless not hard to find examples of simplicial complexes $X$ with diameter equal to two such that $\uH^0(X) = 0$. The idea is to find three or more vertices such that the intersection of the closure of their stars is trivial. One example is provided by the matching complex of $K_5$, the complete graph on five vertices (\emph{cf}. with Proposition~\ref{prop:uberofgraphs}).\\

As a further consequence of Theorem~\ref{thm:0dim} we can explicitly compute $\uH(\Delta^n)$:
\begin{cor}\label{cor:uberofsimplices}
If $n \ge 1$, then
$$\uH^0(\Delta^n) \cong \bigoplus_{k =0}^n \F^{\binom{n+1}{k+1}}_{(k,k+1)}.$$
\end{cor}
\begin{proof}
This follows immediately from the previous result, since for all $1\le i\le n+1$ we have $\overline{St(v_i)} = \Delta^n$; therefore $\uH^0(\Delta^n)$ is generated by all the faces of the simplex, and there are exactly $\binom{n+1}{k+1}$ $k$-dimensional faces in $\Delta^n$.
\end{proof}

It is actually possible to explicitly determine the entire \"ubercomplex and  \"uberhomology of the simplices, revealing the presence of some interesting structures: 
\begin{prop}\label{prop:uberofsimplices}
Let $j,k,n>0$, then $$\left( \ddot{C}^j_*(\Delta^n,k), d \right) \cong \bigoplus_{\substack{\varepsilon \in \Z_2^m\\ |\varepsilon| = j}} \bigoplus^{\binom{n+1 - j}{k}} \left( \widetilde{C}_*(\Delta^{j-1}), \partial \right).$$
If instead $k = 0$, then $$\left( \ddot{C}^j_*(\Delta^n,0), d\right) \cong \bigoplus_{\substack{\varepsilon \in \Z_2^m\\ |\varepsilon| = j}} \left( C_*(\Delta^{j-1}), \partial \right).$$
As a consequence, the only non-trivial \"uberhomology group (other than $\uH^0(\Delta^n)$ which we already covered in Corollary~\ref{cor:uberofsimplices}) is in homological degree $1$
$$\uH^1(\Delta^n) \cong \F_{(0,0)},$$
generated by $\sum_{i = 1}^{n+1} [v_i]_{\varepsilon_i}$. 
\end{prop}
\begin{proof}
As $\Delta^n$ is completely symmetric, all the horizontal homologies $\HH^h(\Delta^n, \varepsilon)$
with the same  $|\varepsilon| = j$ are isomorphic, up to a permutation of labels for the vertices. We can thus restrict our considerations to the homologies $\HH^h(\Delta^n, \hat{\varepsilon}_j)$, where $\hat{\varepsilon}_j (i) = 1$ whenever $i\le j$.

For any fixed $j \in \{1, \ldots,m\}$ there are exactly $j$ black dots in $\hat{\varepsilon}_j$; therefore, for $0<k< m+1-j$, we can consider a subset $S \subseteq \{j+1, \ldots, n+1\}$ and the subcomplex $C(n,k,S) \subseteq \left( C(\Delta^n, \hat{\varepsilon}_j), \partial_h \right)$ spanned by simplices whose white vertices are in $S$. 
As the only white vertices of any simplex are in $S$, the filtration degree of $C(n,k,S)$ is $|S|$, and the lowest dimension degree is $|S|-1$. Note that $C(n,k,S)$ is indeed a subcomplex, as all horizontal differential preserve white vertices. Now, for any $S$ such that $|S| = p$, there is a poset isomorphism between $C(n,k,S)$ and the poset of subsets $S^\prime$ of $\{1, \ldots, j\}$ (including the empty set, that becomes the terminal object of the reduced complex, as in the proof of Theorem~\ref{thm:horizhomoftait}). One example is shown in Figure~\ref{fig:3spxdim}, where each subcomplex $C(n,k,S)$ is a direct summand in the horizontal complexes.

Each of the $C(n,k,S)$ is thus isomorphic to the chain complex computing the reduced homology of $\Delta^{j-1}$ (as $j$ is the maximal number of black vertices available). Furthermore, all of these subcomplexes $C(n,k,S)$ are in direct sum.

If instead we take $k$ to be $0$, we are computing the horizontal homology of the subcomplex of $C(X, \hat{\varepsilon}_j)$ consisting of simplices with all black vertices; this is clearly isomorphic to $\left(C_*(\Delta^{j-1}), \partial \right)$. 

\begin{figure}
\centering
\includegraphics[width=12cm]{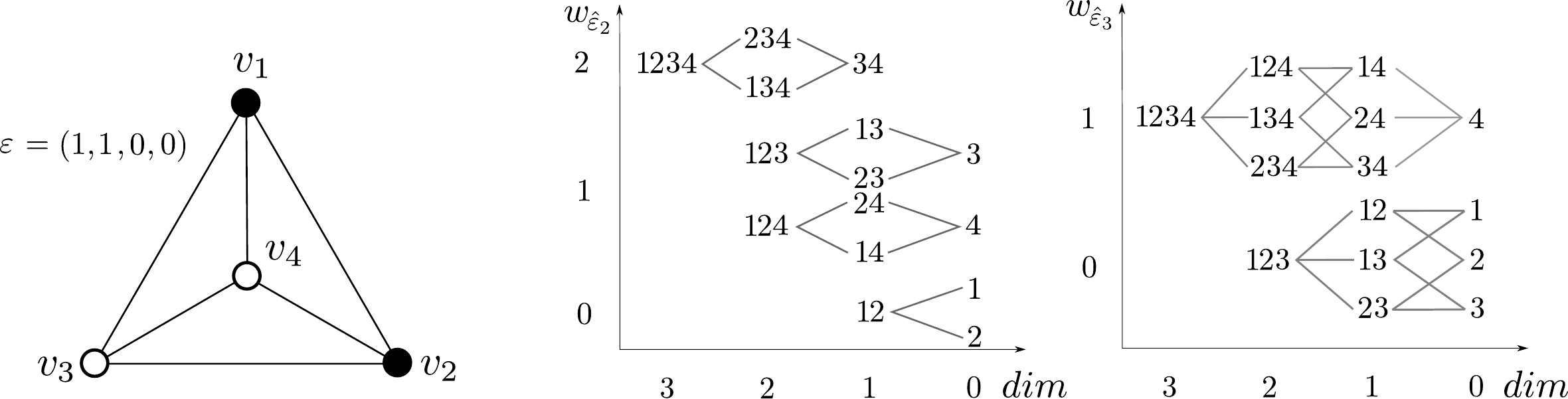}
\caption{From the left: $\Delta^3$ coloured with  $\hat{\varepsilon}_2$, and its horizontal complex. Note that on each row we can see the simplicial complex for the reduced homology of $\Delta^1$, except for the lowest row in which the complex is not reduced. On the right the horizontal complex of $(\Delta^3, \hat{\varepsilon}_3)$. 
}
\label{fig:3spxdim}
\end{figure}

Using this we can conclude: the only non-trivial elements in the horizontal homologies of $(\Delta^n, \hat{\varepsilon}_j)$ for $0<j\le n+1$ (and therefore the only elements in $\ddot{C}^j(\Delta^n)$) can be identified with the generators of the homology of $\Delta^{j-1}$ in filtration degree $0$. 
In particular, we have a single generator in bi-degree $(0,0)$ for each $\varepsilon$ with $|\varepsilon|>0$. So the \"ubercomplex is isomorphic to the reduced simplicial chain complex of the boundary of the $n$-simplex (here the terminal element comes from the generator of $\ddot{C}^{n+1}(\Delta^n) \cong \HH^h(\Delta^n) \cong \HH_*(\Delta^n)$). Hence, the homology is non-trivial only in degree $j = 1$, generated by $\sum_{p = 1}^{n+1} [v_i]_{\varepsilon_i}$ as required.
\end{proof}
\begin{cor}\label{cor:uberofspheres}
The \"uberhomology of the simplicial spheres $\partial \Delta^n$ are as follows:
\begin{equation}\label{eqn:sphere0dim}
\uH^0 (\partial\Delta^n) = \bigoplus_{p = 0}^{n-2} \F^{\binom{n+1}{p+1}}_{(p, p+1)},
\end{equation}

while for $0<j<n+1$ 
\begin{equation}\label{eqn:sphereotherdim}
\uH^j (\partial\Delta^n) =  \F^{\binom{n+1}{j}}_{(n-1, n+1-j)}.
\end{equation}
\end{cor}
\begin{proof}
This is a consequence of Theorem~\ref{thm:0dim} and the computations in the proof of Proposition~\ref{prop:uberofsimplices}. For the \"uberhomology in degree $0$ just observe that the intersection of the closed stars of vertices in $\partial\Delta^n$ consists of its $(n-1)$-skeleton. These simplices are those appearing in Equation~\eqref{eqn:sphere0dim}.

For the remaining part, we only need to perform the same computation as in the proof of Proposition~\ref{prop:uberofsimplices}, after removing the top dimensional simplex. This removal has the effect of adding a generator in filtration degree $j$ in each horizontal homology. More precisely, if $|\varepsilon|=j$, this new generator in $\HH^h(\partial \Delta^n, \varepsilon)$ is given by the sum of the maximal simplices in $\partial \Delta^n$ containing all the white vertices. All these simplices are in bi-degree $(n-1,n+1-j)$, and there are $\binom{n+1}{j}$ of them for each fixed $j$, one for each $\varepsilon$ with $|\varepsilon| = j$. 
From this description it is apparent that all these generators are in the kernel of any component $d_\eta$ of the differential starting at $\varepsilon$.
\end{proof}

\begin{rmk}
Using Theorem \ref{thm:0dim}, we can provide an example of a pair of distinct simplicial complexes with isomorphic \"uberhomologies in degree $0$, namely the $1$-simplex $\Delta^1$ and the union of two $2$-simplices attached along a codimension one face. 

The same holds for any number of simplices of the same dimension glued along a common codimension $1$ face; hence we can also easily produce an infinite number of simplicial complexes with the same non-trivial $\uH^0$ and an arbitrary number of vertices.
\end{rmk} 

It is generally harder to give an a priori description of the top dimensional \"uberhomology groups; using the following results, it is nonetheless possible to say something meaningful in some favourable cases. Recall that $\overline{\varepsilon_i}$ denotes the colouring complementary to $\varepsilon_i$.
\begin{prop}\label{prop:almosttop}
Let $X$ be a finite simplicial complex with vertices $V(X) = \{v_1, \ldots, v_m\}$. Then there are isomorphisms
\begin{equation*}
\HH^h_*(X, \overline{\varepsilon_i},k)    \cong \begin{cases}
\widetilde{\HH}_{*-1}(lk(v_i)) & \text{ if } k = 1\\
\HH_*(X \setminus St(v_i)) & \text{ if }  k = 0,
\end{cases}
\end{equation*}
and these are the only non-trivial groups. 

In particular, if $X$ is a triangulation of a closed $n$-dimensional manifold, then 
\begin{equation*}
\HH^h_*(X, \overline{\varepsilon_i},k)    \cong \begin{cases}
\widetilde{\HH}_*(S^{n-1}) \cong \F_{(n-1, 1)} & \text{ if } $k = 1$\\
\HH_*(|X| \setminus B^n) & \text{ if }  $k = 0$
\end{cases}
\end{equation*}
where $B^n$ is an open ball.
\end{prop}
\begin{proof}
Let us start by observing that for the colourings $\overline{\varepsilon_i}$ the only two filtration degrees in which $\HH^h$ can be supported are $0$ and $1$. 
Moreover, the complex in $w_{\overline{\varepsilon_i}}$ degree $1$ is generated by simplices in $C(X)$ that contain the vertex $v_i$. This complex is isomorphic to the complex for $lk(v_i)$; more formally, there is a bijection between $C(lk(v_i))$ and $C(X, \overline{\varepsilon_i}, 1)$ taking the simplex $\sigma \subseteq lk(v_i)$ to $\langle V(\sigma), v_i \rangle$, and the empty simplex to $\langle v_i\rangle$ (see Figure~\ref{fig:almostallblack}).

\begin{figure}
\centering
\includegraphics[width = 9cm]{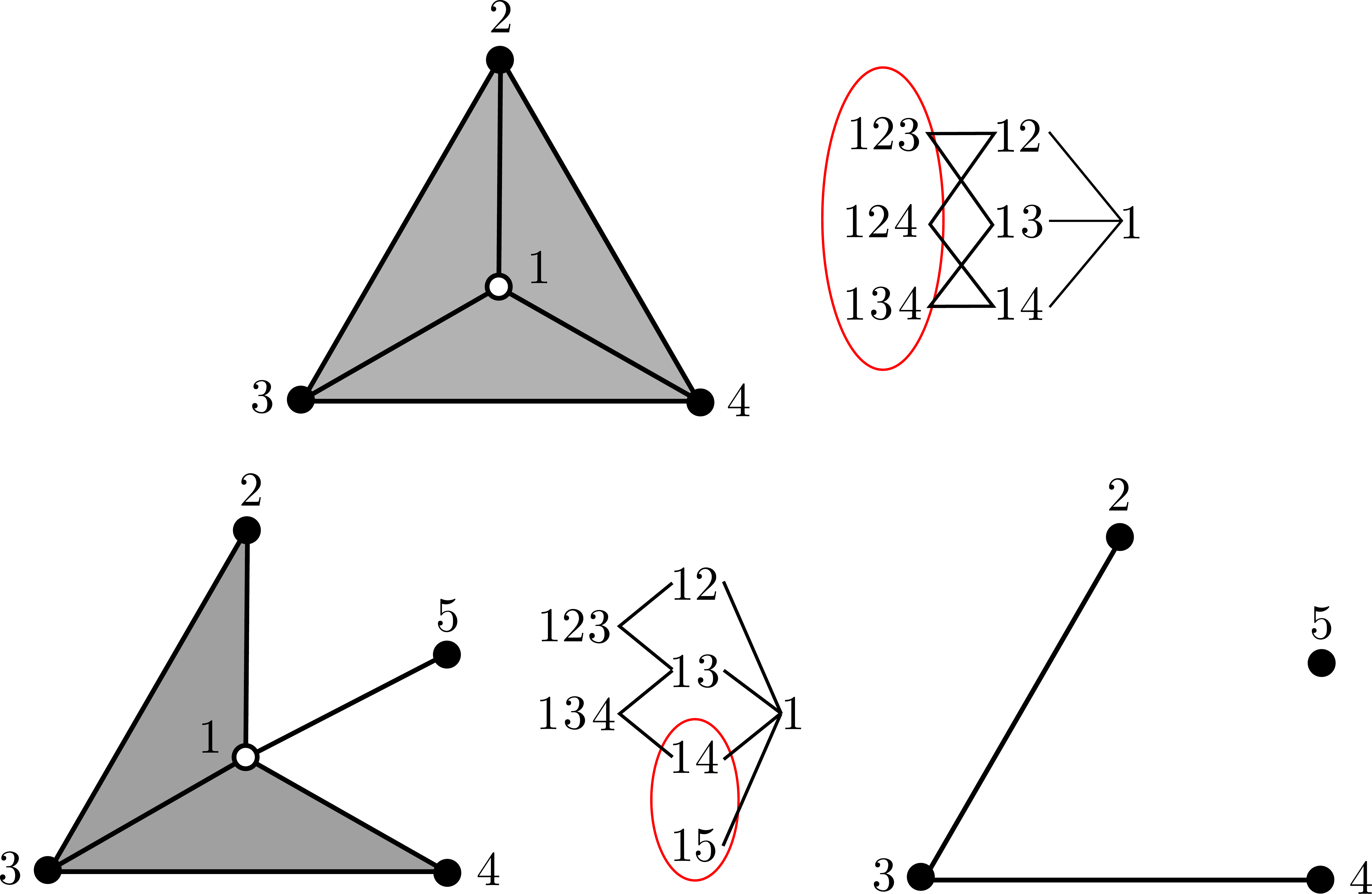}
\caption{On the top the simplicial complex obtained by gluing $3$ simplices of dimension two, with colouring $(0,1,1,1))$, and the horizontal chain complex in filtration degree $1$. Generators circled in red sum up to nontrivial generators of $\HH^h$. Similarly, below for the complex obtained by attaching two $2$ simplices and one $1$ simplex. On the right the link of the unique white vertex. Consistently with Proposition~\ref{prop:almosttop} in both case we are computing $\widetilde{\HH}(lk(v_1))$.}
\label{fig:almostallblack}
\end{figure}
In filtration level $0$ instead we find all the simplices that do not contain $v_i$, which are precisely those in  $X \setminus St(v_i)$. Since there are no white vertices at play here, the horizontal differential reduces to the simplicial differential on $X \setminus St(v_i)$. 

The second part of the proposition follows by recalling that for a homology manifold $lk(v_i)$ is by definition a sphere.
\end{proof}

We can now prove that the top-dimensional degree of the \"uberhomology of homology manifolds is a topological invariant.
\begin{thm}\label{thm:topdim}
If $X$ is a $m$-vertex triangulation of a closed connected $n$-manifold, then $$  \uH^m (X) \cong \F_{(n,0)},$$ generated by the fundamental class of $|X|$.
\end{thm}
\begin{proof}
Since $d^m$ is the $0$ map,  $\uH^m (X)$ is just $\HH_n(X)$, modulo the union of images of $d_{\overline{\eta}_i}$. Recall that $\overline{\eta}_i$ and $\overline{\varepsilon}_i$ are the edge (respectively vertex) of the cube whose only $*$ (respectively $0$) component is in the $i$-th position. Moreover, all the groups $\HH^h_*(X, \overline{\varepsilon}_i)$ are contained in only two filtration levels, $0$ and $1$. By Proposition~\ref{prop:almosttop}, the image of $d_{\overline{\eta}_i}$ is isomorphic to $\HH_*(X \setminus B^n) \cong \HH_*(X)$ for $*<n$, and $0$ in homological degree $n$. Hence the only homology class in $\HH^h(X, (1,\ldots,1)) \cong \HH (X)$ that does not belong to the image of $d^{m-1}$ is $[X] \in \HH_n(X)$, the fundamental class of $|X|$.
\end{proof}

This next result implies that the \"uberhomology in its extremal degrees is well-behaved under the operation of taking the cone or suspension.
Let us denote by $Cone(X)$ and $\Sigma(X)$ the cone and suspension of the simplicial complex $X$, endowed with the obvious simplicial structure induced by $X$. By convention, the cone and suspension of the empty set are given by one and two points respectively. Recall that the cone of any simplicial complex is contractible, and that $\widetilde{\HH}_i(X) \cong \widetilde{\HH}_{i+1}(\Sigma (X))$. 

In the next result we will slightly abuse the notation (consistently with Theorem~\ref{thm:0dim}), denoting by $Cone(\uH^0(X))$ the free group generated over $\F$ by the simplices in the cone of $\bigcap_{i=1}^m \overline{St(v_i)}$.

\begin{prop}\label{prop:cone}
Let $X$ be a simplicial complex with $m$ vertices. Then 
\begin{equation}\label{eqn:cone}
\uH^{m+1}(Cone(X)) = 0, \;\; \uH^0(Cone(X)) \cong 
Cone(\uH^0(X)).
\end{equation}

For the suspension instead we have 
\begin{equation}\label{eqn:suspension}
\uH^0(\Sigma(X)) \cong \uH^0(X)
\end{equation}
and if furthermore $X$ is the triangulation of a closed $n$-dimensional manifold,
\begin{equation}
\uH^{m+2}(\Sigma(X)) \cong \uH^{m}(X)  \cong \F_{(n+1,0)}.
\end{equation}
\end{prop}
\begin{proof}
We first prove the two statements in Equation~\eqref{eqn:cone}, starting from the right one. Call $\hat{v}$ the cone vertex; clearly $\overline{St(\hat{v})} = Cone(X)$, while for all other vertices $v$,  $\overline{St(v)}_{Cone(X)} = Cone(\overline{St(v)}_X)$ (here the subscripts are used to distinguish the ambient space of the star). The result then follows from Theorem~\ref{thm:0dim}, after noting that the cone vertex is contained in the closed star of any other vertex.
To prove the vanishing of $\uH^{m+1}(Cone(X))$, using Theorem~\ref{prop:almosttop} we only need to look at the subcomplex of $\HH^h(Cone(X), \overline{\varepsilon_i})$ in filtration level $0$ (as these are the only ones contributing to the image of the differential $d^{m-1}$). It is immediate to realise that if $\varepsilon$ is the colouring where the only white vertex is $\hat{v}$, then  $\HH^h_*(Cone(X), \varepsilon,0)$ is isomorphic to $\HH_*(X)$; in particular the corresponding map $d_\eta$ is a surjection on $\HH^h_*(Cone(X), (1,\ldots, 1)) \cong \HH_*(Cone(X)) \cong \F$ in homological degree $0$, as the cone is contractible.

Now for the statement in Equation~\eqref{eqn:suspension}, simply observe that, if $\hat{v}$ and  $\hat{w}$ are the two cone points in $\Sigma(X)$, we have $\overline{St(\hat{v})} \cap \overline{St(\hat{w})} = X$. The results then follow directly from Theorem~\ref{thm:0dim}, while the last statement follows from Theorem~\ref{thm:topdim}.
\end{proof}

\section{\"Uberhomology and graphs}\label{sec:ubergraphs}

In this section we specialise the definition of the \"uberhomology to connected graphs.
In order for this to make sense, we will of course only consider simple graphs, \emph{i.e.}~$1$-dimensional simplicial complexes. By further specialising to single bi-gradings (as in Remark~\ref{rmk:splitting_degrees}) we will obtain four distinct homology theories which we denote by $\mathbb{H}^0(G), \mathbb{H}^1_0(G), \mathbb{H}^1_1(G)$ and $\mathbb{H}^2 (G)$. These homologies exhibit qualitatively different behaviours, and might be of independent interest. We point out that the results in this section show that the information contained in these homology groups is not related to the one obtained in Section~\ref{sec:graphdiss} by collecting the horizontal homologies of a graph in the $\Theta$ invariants. \\

Let us begin by examining the simplest of these homologies, $\mathbb{H}^0(G)$, defined as the \"uberhomology of $G$ in bi-degree (dimension of simplices, $\varepsilon$-weight) =  $(0,0)$; recall the structure theorems for the horizontal homology of a graph from Section~\ref{sec:graphdiss} (in the discussion around Equation~\eqref{eqn:horgraph}); note that the restriction of the maps $d_{\eta}$ to $\HH^h_1(G, \varepsilon , 0)$ are all injective (as adding further black points can only increase the rank of $\HH_1(Bl(G,\varepsilon))$) thus  $\uH_*(G, 0)$ is a singly-graded homology, which we'll denote by $\mathbb{H}^0(G)$.
The generators in the complex of $\mathbb{H}^0(G)$ are the connected components of $Bl(G,\varepsilon)$ (see Figure~\ref{fig:H0path}).

We can give an equivalent definition of $\mathbb{H}^0(G)$ that does not involve any filtration-preserving differential $d$; given $G$, 
consider the poset of subsets of its vertices, ordered by inclusion. To each pair of comparable subsets differing by exactly one element --corresponding to an inclusion $Bl(G,\varepsilon) \subset Bl(G ,\varepsilon^\prime)$-- we associate the map $$(\iota_\eta)_* : \F\langle \pi_0(Bl(G,\varepsilon)) \rangle \longrightarrow \F\langle\pi_0(Bl(G,\varepsilon^\prime)) \rangle $$ induced by the inclusion $\iota_\eta$ of $\pi_0(Bl(G,\varepsilon))$ in $\pi_0(Bl(G,\varepsilon^\prime))$ (here $\pi_0$ is the set of connected components).
Then for $j = 1, \ldots, m$ call $$\mathcal{C}^j (G) = \bigoplus_{\substack{\varepsilon \in \Z_2^m \\ |\varepsilon| = j}} \F \langle \pi_0 (Bl(G, \varepsilon)) \rangle, \text{ and } \mathbb{H}^0(G) = \mathrm{H}_*\left(\mathcal{C}^j (G),\bigoplus_{|\eta| = j} (\iota_{\eta})_*\right).$$

The identification of $\mathbb{H}^0(G)$ with the \"uberhomology of $G$ in bi-degree $(0,0)$ can be sketched as follows; by definition, $\uH^*_{0}(G,0)$ is generated by horizontal $0$-homology classes spanned by black vertices for varying bi-colourings of $G$. Therefore $\mathcal{C}^j(G)$ can be identified with the submodule of $\ddot{C}^j (G)$ generated by connected components of black subgraphs. Furthermore, as the filtration degree can not decrease (being already minimal), the differentials in the two complexes agree.

\begin{prop}
If $G = K_m$, then $\mathbb{H}^0(K_m) \cong \F_{(1)}$, generated by the sum of all the subgraphs consisting of a unique black vertex.
\end{prop}
\begin{proof}
Note that for every $\varepsilon$, the black subgraph $Bl(G, \varepsilon)$ is connected; hence the chain complex looks like a $m$-cube (as in Figure~\ref{fig:cubo}) with a copy of $\F$ on each vertex, except the vertex corresponding to $\varepsilon = (0, \ldots, 0) $ which is removed, since we have $Bl(G,(0, \ldots, 0)) = \emptyset$ (therefore the complex is isomorphic to the simplicial chain complex computing the cohomology of the $m$-simplex). Each remaining vertex is endowed with a copy of $\F$ and the edges are all isomorphisms. The homology of this complex is generated by the sum of all elements in the lowest non-trivial degree, which consists of the sum of the connected components of the subgraphs induced by elementary colourings.
\end{proof}

\begin{figure}[ht]
\centering
\includegraphics[width = 12cm]{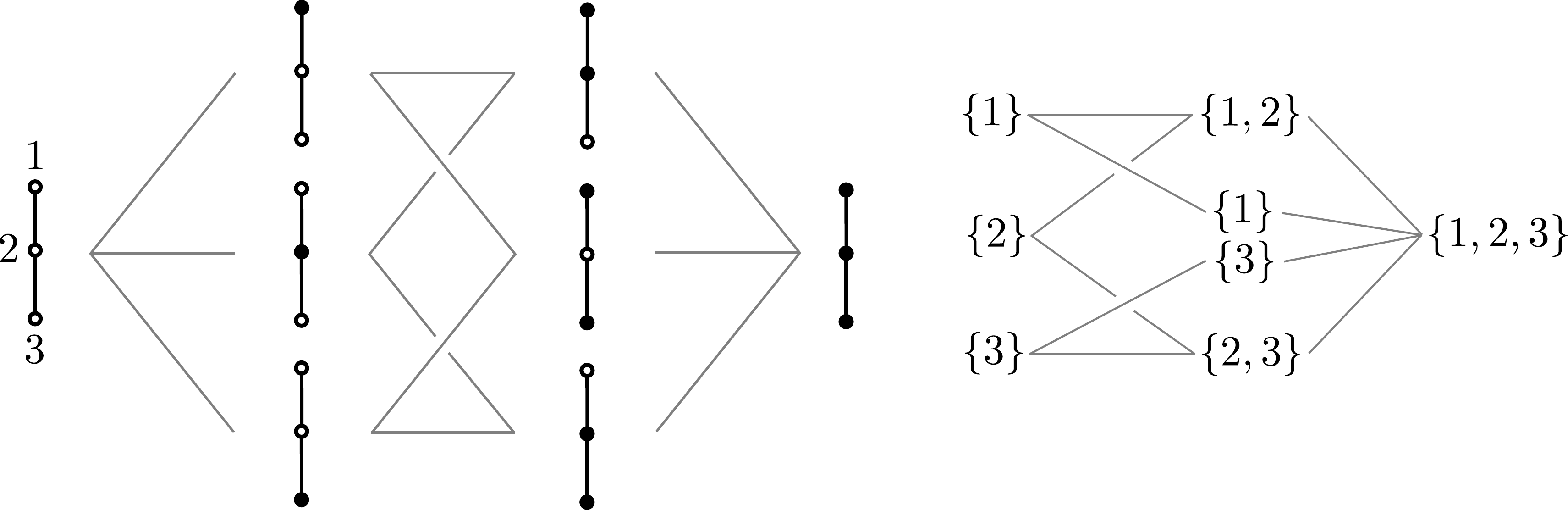}
\caption{The poset structure for the length $2$ path $P_2$, together with the complex whose homology is $\mathbb{H}^0(P_2) = 0$. Brackets denote connected components.}
\label{fig:H0path}
\end{figure}

A Sage program computing $\mathbb{H}^0 (G)$ is available at \cite{miogithub}
.
With this program it is possible to verify that, for the graphs $G_1, G_2$ from Figure~\ref{fig:twographs}, we get $\mathbb{H}^0(G_1) \cong \F_{(2)}$ and $\mathbb{H}^0(G_2) \cong \F_{(1)}$. Therefore, somewhat surprisingly, $\mathbb{H}^0$ can pick up some subtle information about graphs.

Despite being extremely simple, this homology exhibits some unexpected behaviours; in what follows, $Cube(n)$ denotes the $1$-skeleton of the $n$-dimensional cube, $Grid(n,m)$ is the plane square grid of size $(m,n)$, and $K_{m,n}$ is the complete bipartite graph on $(m,n)$ vertices.

\begin{equation*}
\mathbb{H}^0(Cube(n)) = \begin{cases}
\F_{(2)} & \text{ if } n=2\\
\F_{(4)}^3 & \text{ if } n=3\\
\F_{(8)}^{21} & \text{ if } n=4
\end{cases}
\end{equation*}

The only pairs $(m,n)$ with $mn \le 16$ such that $\mathbb{H}^0(Grid(n))$ is non-trivial are the following ones:
\begin{equation*}
\mathbb{H}^0(Grid(m,n)) = \begin{cases}
\F_{(0)} & \text{ if } \{m,n\} = \{1,2\}\\
\F_{(1)} & \text{ if } \{m,n\} = \{2,2\}\\
\F_{(5)} & \text{ if } \{m,n\} = \{3,3\}\\
\F_{(7)} & \text{ if } \{m,n\} = \{3,4\}\\
\F_{(9)} & \text{ if } \{m,n\} = \{3,5\}\\
\F_{(9)} & \text{ if } \{m,n\} = \{4,4\}
\end{cases}
\end{equation*}

Aided by our program, we verified the following conjecture for graphs with up to $10$ vertices:
\begin{conj}\label{conj:0uberforgraphs}
These isomorphisms hold in general
\begin{itemize}
\item $\mathbb{H}^0(G) = 0$ if $G$ is a tree.
\item $\mathbb{H}^0(K_{m,n}) \cong \F_{(2)}$ for all $m,n >1$.
\item $\mathbb{H}^0(L_m) \cong \F_{(m-2)}$, where the generator is the sum of the two connected components in the colouring $\varepsilon = (0, 1, \ldots, 1, 0, 1)$
\end{itemize}
\end{conj}
\begin{rmk}
The first and last item of the previous conjecture have since been proved in~\cite{ubergrafi}.
\end{rmk}

We can also consider the \"uberhomology of $G$ in other (bi-)degrees; if we define $\mathbb{H}^2_*(G)$ as $\uH^*_1(G,2)$, we see that $\mathbb{H}^2(G)$ can be defined as the homology whose generators are edges with both endpoints coloured in white, placed along the poset $\Z_2^{|V(G)|}$, with differentials induced by inclusion. 
Computations with our program seem to imply that $\mathbb{H}^2(G)$ is trivial:
\begin{conj}
If $G$ is a simple and connected graph, then $\mathbb{H}^2(G) = 0$.
\end{conj}

Finally, in analogy with the previous ones, we can define the homologies $\mathbb{H}^1_0(G)$ and $\mathbb{H}^1_1(G)$ as the \"uberhomologies of $G$ in filtration degree $1$, and in homological degree $0$ and $1$ respectively. 
It is then easy to see that in degree $k$ the group $\mathbb{H}^1_0(G)$  is generated by those white vertices in $(G,\varepsilon)$ such that all of their neighbours are white as well, and $|\varepsilon| = k$. 
We can use this description to compute $\mathbb{H}^1_0$ for the complete graphs: 
\begin{equation}
\mathbb{H}^1_0 (K_m) \cong \F^m_{(0)}.
\end{equation}
On the other hand $\mathbb{H}^1_0$ appears to be trivial for many classes of graphs, such as loops and trees with more than $4$ vertices (but is highly non-trivial for graphs obtained by removing a few edges from $K_m$). In particular there might be some relation of sorts between the rank of $\mathbb{H}^1_0 (G)$ and the connectivity of $G$.\\

The description of $\mathbb{H}^1_1 (G)$ is slightly more complex, due to the fact that, unlike with the other homologies we just defined, here we must keep track of the horizontal differential.
Each horizontal homology that belongs to the complex with homology $\mathbb{H}^1_1 (G)$ is as in the rightmost summand in Equation~\eqref{eqn:horgraph}. More in detail, as shown in Figure~\ref{fig:varieubergrafi}, we get $p-1$ generators in $\HH^h_1(G, \varepsilon,1)$ for each white vertex in $G$ whose link contains $p$ black vertices, if $p>0$, and $0$ otherwise. 
\begin{figure}[ht]
\centering
\includegraphics[width = 6cm]{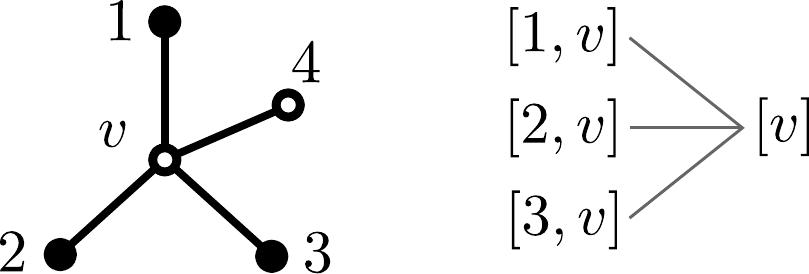}
\caption{The neighbourhood of the vertex $v$ containing $3$ black vertices; this produces the complex on the right, whose homology has rank $2$.}
\label{fig:varieubergrafi}
\end{figure}

It is possible to show that $\mathbb{H}^1_1$ is not a trivial homology; as a small example, one can compute $\mathbb{H}^1_1 (L_4) = \F^4_{(2)}$.

~\\
Here, we use results from Sections~\ref{sec:uberhomology} and~\ref{sec:computations} to compute the \"uberhomology in its extremal homological degrees for loop graphs.

\begin{rmk}\label{rmk:uberloops}
Consider the $1$-dimensional simplicial complex $L_m$, the loop of length $m \ge 3$. Then by using Theorem~\ref{thm:0dim} and Theorem~\ref{thm:topdim} we can easily deduce that (see also Figure~\ref{fig:uberofloop})
$$\uH^0(L_m) = 0 \Longleftrightarrow m > 3,$$
and
$$\uH^m(L_m) \cong \F_{(1,0)}.$$
\begin{figure}[ht]
\centering
\includegraphics[width=10cm]{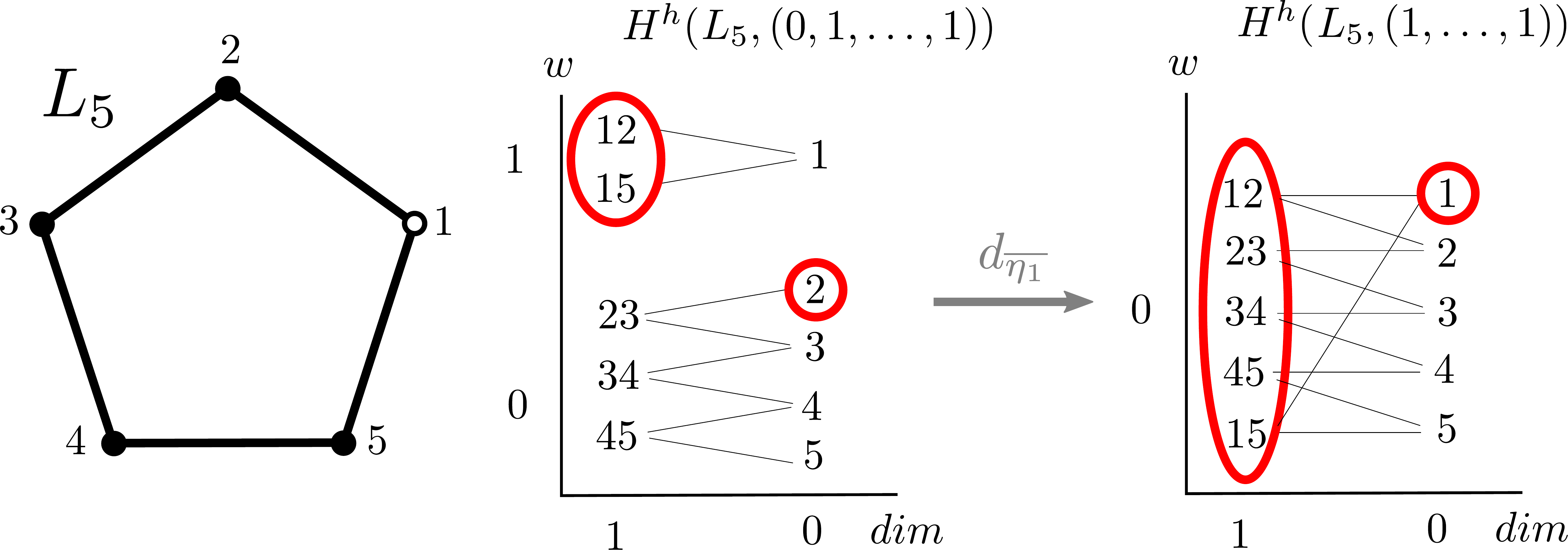}
\caption{On the left, the coloured complex $(L_5, (0,1,1,1,1))$; on the right, the two horizontal complexes related by the map $d_{\overline{\eta_i}}$. Generators of the horizontal homology are circled in red.}
\label{fig:uberofloop}
\end{figure}
\end{rmk}


In a somewhat different direction, there are several ways of extracting properties of a graph $G$ by associating a simplicial complex to it (see \emph{e.g.}~\cite{jonsson2007simplicial} for a survey); one example is the matching complex defined in Section~\ref{sec:graphs}. We can conclude with a final construction, loosely related to the content of Section~\ref{sec:graphs}, where the  \"uberhomology appears ``naturally''.

\begin{exa}
Let $G$ be a connected graph with edges $\{e_1, \ldots, e_m\}$. Consider the $1$-skeleton of the $m$-dimensional cube with vertices in $\Z_2^m$; for each such vertex $\varepsilon$ of this cube we can colour the $i$-\emph{th} edge of the graph $G$ black or white according to whether  $\varepsilon(i)$ is $1$ or $0$ (see Figure~\ref{fig:subgraphs}). 
\begin{figure}[ht]
\centering
\includegraphics[width = 9cm]{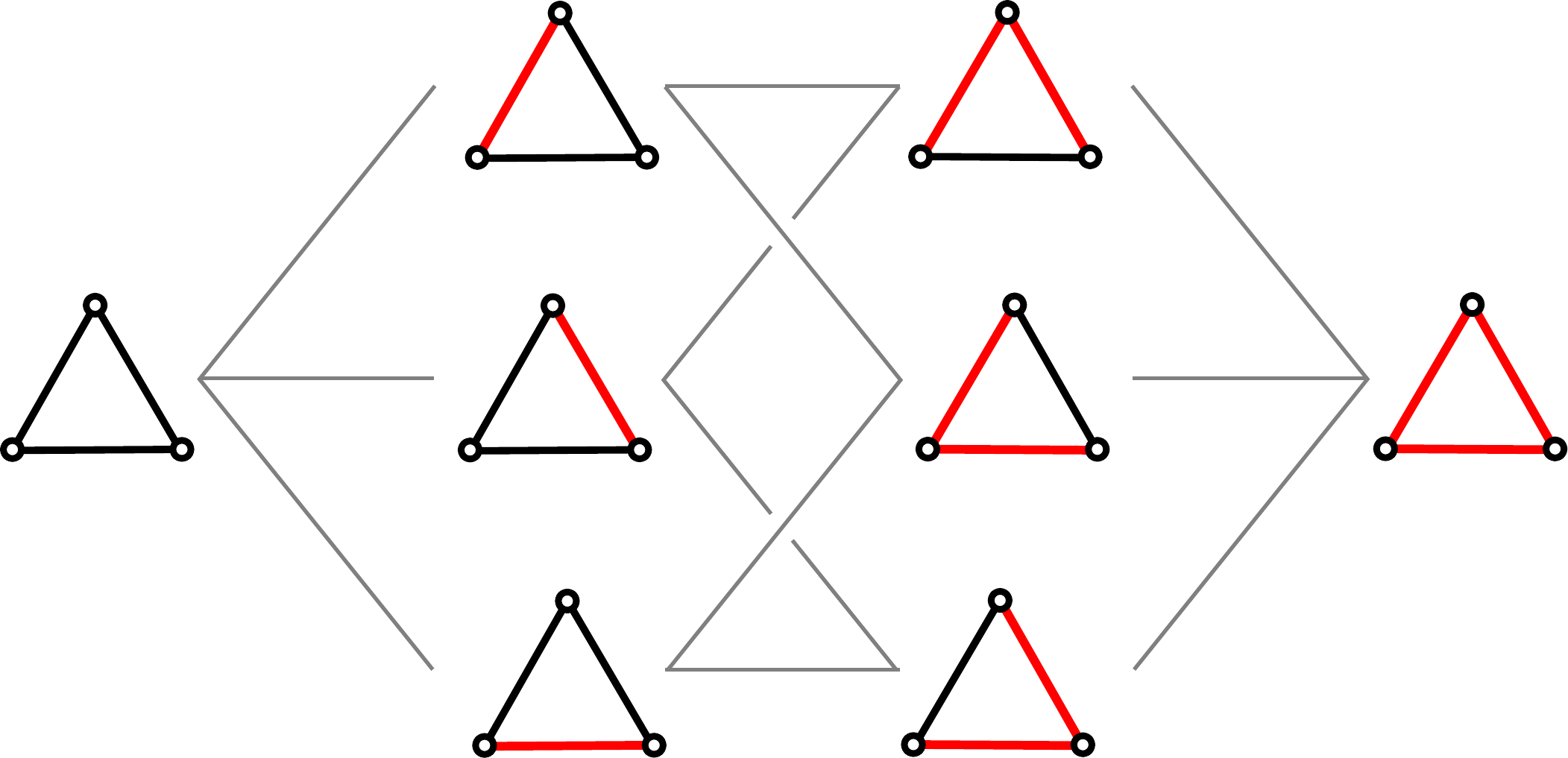}
\caption{The poset structure of the edges in the cycle graph $C_3$. }
\label{fig:subgraphs}
\end{figure}

Thus for each $\varepsilon$ we can bi-colour the matching complex $\Mat (G)$ accordingly. 
Now we can decorate the $m$-cube with $\HH^h(\Mat (G), \varepsilon)$, and define the \"uberhomology $\uH(\Mat (G))$. 
\end{exa}

It would be of course interesting to see which properties of $G$ (and $\Mat (G)$) are reflected in the structure of $\uH(G)$ and  $\uH(\Mat (G))$. The following result provides some information in the $0$-th degree of the \"uberhomology.
\begin{prop}\label{prop:uberofgraphs}
Consider a simple graph $G$ with edges $E(G) = \{e_1, \ldots, e_m\}$. Then the generators of  $\displaystyle \bigoplus_{k >0} \HH^h_*(\Mat (G),\varepsilon_i, k)$  are in bijection with matchings that do not contain $e_i$, and thus $\uH^0(\Mat (G)) \cong 0$. 

Furthermore, $$\HH^h_*(\Mat (G), \overline{\varepsilon_i},0) \cong \HH_*(\Mat (G\setminus e_i)),$$ that is, we obtain the homology of the matching complex of the graph $G$ with the $i$-th edge removed. 
\end{prop}
\begin{proof}
The first part of the statement follows from the proof of Theorem~\ref{thm:dmmandsubgraphs}: for an elementary colouring we get a Morse matching $I(\Mat (G), \varepsilon_i)$ that pairs all the simplices containing $e_i$ with their unique maximal face that does not contain $e_i$. The only exception occurs for the vertex determined by $e_i$, which is the only simplex in filtration degree $0$. The triviality of the \"uberhomology in degree $0$, then follows by Theorem~\ref{thm:0dim}, while the last statement is an immediate consequence of Proposition~\ref{prop:almosttop}.
\end{proof}

\begin{rmk}
In the opposite direction, it is generally hard to compute $\uH^{|E(G)|}(\Mat (G))$, except when $\Mat (G)$ is a homology manifold (using Theorem~\ref{thm:topdim}). The graphs for which this occurs were completely classified in \cite{bayer2020manifold}.
\end{rmk}

~\\
There are many questions regarding the type and amount of information on $X$ one can hope to extract from $\uH(X)$; we collect here some reasonable questions on the matter.
\begin{itemize}
\item We have seen in Theorems~\ref{thm:0dim},~\ref{thm:topdim} that the uberhomology in its extremal homological degrees admits alternative interpretations.  What kind of information is encoded in $\uH(X)$ in the intermediate degrees? 
\item Is it possible to construct a ``spacification'' of $\uH(X)$, \emph{i.e.}~a topological space $Y$ (or more likely a collection of such spaces) such that $\mathrm{H}_*(Y) = \ddot{\mathrm{H}}(X)$? 
\item Is it possible to define the usual algebraic topology toolbox available for simplicial homology in the case of the \"uberhomology (functoriality, exact sequences, K\"unneth formula, etc.)?
\item  What is the decategorification of $\uH(X)$? Is it related --or does it specialise to-- any known polynomial invariants of simplicial complexes? 
\end{itemize}

\end{document}